\documentclass{amsart}
\usepackage{amssymb}
\usepackage{amsmath}
\usepackage[dvipsnames]{xcolor}
\definecolor{darkgreen}{rgb}{0.0, 0.5, 0.0}
\usepackage{graphicx}
\usepackage{manfnt}
\usepackage{amsthm}
\usepackage[mathscr]{euscript}
\usepackage{mathtools}
\usepackage{epstopdf}
\usepackage{bigints}
\usepackage{enumerate}
\usepackage{enumitem}
\graphicspath{ {./images/} }
\usepackage{tikz-cd}
\usepackage{bm}
\usetikzlibrary{arrows}
\usetikzlibrary{positioning}
\usetikzlibrary{decorations.pathreplacing}
\usepackage{mathrsfs}
\numberwithin{equation}{section}
\usepackage{hyperref}
\usepackage{multicol}
\usepackage{mathrsfs} 
\usepackage{braket}
\usepackage{mathdots}
\usepackage{bbm}
\usepackage[table,xcdraw]{xcolor}
\usepackage{float}

\newcommand{\bbm}{\begin{bmatrix}}
\newcommand{\ebm}{\end{bmatrix}}
\newcommand{\bv}{\begin{vmatrix}}
\newcommand{\ev}{\end{vmatrix}}

\newcommand{\C}{\mathbb{C}}
\newcommand{\Z}{\mathbb{Z}}

\newcommand{\R}{\mathbb{R}}
\newcommand{\beq}{\begin{equation*}}
\newcommand{\eeq}{\end{equation*}}
\newcommand{\beqn}{\begin{eqnarray*}}
\newcommand{\eeqn}{\end{eqnarray*}}

\newcommand{\mf}{\mathfrak}
\newcommand{\mc}{\mathcal}
\newcommand{\bp}{\begin{pmatrix}}
\newcommand{\ep}{\end{pmatrix}}
\newcommand{\SBim}{\mathbb{S}\mathrm{Bim}}

\newcommand{\op}[1]{\operatorname{#1}}

\DeclareMathOperator{\End}{End}
\DeclareMathOperator{\Hom}{Hom}

\DeclareMathOperator{\SL}{SL}

\DeclareMathOperator{\SU}{SU}
\DeclareMathOperator{\Sp}{Sp}
\DeclareMathOperator{\GL}{GL}

\theoremstyle{plain}
\newtheorem{theorem}{Theorem}[section]
\newtheorem*{theorem*}{Theorem}
\theoremstyle{plain}
\theoremstyle{definition}
\newtheorem{definition}[theorem]{Definition}
\newtheorem{lemma}[theorem]{Lemma}
\theoremstyle{definition}

\theoremstyle{plain}

\theoremstyle{plain}

\theoremstyle{remark}
\newtheorem{remark}[theorem]{Remark}
\theoremstyle{remark}
\newtheorem{notation}[theorem]{Notation}
\theoremstyle{definition}

\theoremstyle{definition}

\title{Lusztig--Vogan categories of equal rank 2}
\author{Daniel Dunmore, Anna Romanov, and Victor L. Zhang}
\date{}

\address{D.~Dunmore: University of New South Wales (UNSW), School of Mathematics and Statistics, Australia;
  \href{https://delphinoid.github.io/}{{\ttfamily\upshape delphinoid.github.io}};
  \href{https://orcid.org/0009-0004-9386-7662}{{\ttfamily\upshape ORCID: 0009-0004-9386-7662}}
}
\email{d.dunmore@unsw.edu.au}

\address{A.~Romanov: University of New South Wales (UNSW), School of Mathematics and Statistics, Australia;
  \href{https://web.maths.unsw.edu.au/~aromanov/}{{\ttfamily\upshape web.maths.unsw.edu.au/\~{}aromanov}};
  \href{https://orcid.org/0000-0000-0000-0000}{{\ttfamily\upshape ORCID: 0000-0002-9146-7217}}
}
\email{a.romanov@unsw.edu.au}

\address{V.L.~Zhang: University of New South Wales (UNSW), School of Mathematics and Statistics, Australia;
  \href{https://dustbringer.github.io/}{{\ttfamily\upshape dustbringer.github.io}};
  \href{https://orcid.org/0009-0007-5799-6477}{{\ttfamily\upshape ORCID: 0009-0007-5799-6477}}
}
\email{victor.l.zhang@unsw.edu.au}

\begin{document}

\maketitle

\begin{abstract}
    Lusztig--Vogan categories are categorifications of the principal block of the Lusztig--Vogan module over the Hecke algebra, which captures information about characters of irreducible admissible representations of a real reductive group. Lusztig--Vogan categories can be constructed as module categories over Soergel bimodules. In this paper, we describe the structure of the rank 2 Lusztig--Vogan categories corresponding to equal rank real groups. More precisely, we classify indecomposable objects and describe the action of generating Soergel bimodules, recovering the $W$-graph of the underlying Lusztig--Vogan module. We also provide an algorithm which completes this procedure for arbitrary finite rank Lusztig--Vogan categories, including those which do not correspond to a real reductive group. 
\end{abstract}
\tableofcontents

\section{Introduction}
\label{introduction}

\subsection{Overview}
\label{sec: overview} 

Let $G$ be a complex connected reductive algebraic group, $\theta: G \rightarrow G$ a holomorphic involution, and $K \subseteq G^\theta$ a finite index subgroup of the fixed point group. Associated to this data, there are two important objects:
\begin{enumerate}
    \item a real reductive group $G_\R$, and 
    \item a Lusztig--Vogan module $M_{LV}$ over the Hecke algebra of $G$. 
\end{enumerate}
These objects are closely related: the module $M_{LV}$ can be used to construct a family of polynomials, appearing as entries in a change-of-basis matrix, which give character formulas for irreducible admissible representations of $G_{\R}$ \cite{LV, Vogan4}. 

In \cite{LarsonRomanov}, a categorification of the principal block\footnote{The block of $M_{LV}$ containing the trivial representation.} of $M_{LV}$ is constructed as a module category over the monoidal category of Soergel bimodules. (See also \cite{BV}, where an equivalent category appears as sheaves on the block variety.) In the case where the real group is of equal rank\footnote{Meaning that $G$ and $K$ share a maximal torus.}, the only input for the construction is the Weyl group of $K$, viewed as a reflection subgroup of the Weyl group of $G$. This results in a rich new source of module categories over Soergel bimodules, whose structure is controlled by the admissible representation theory of a real reductive group.

In fact, Lusztig--Vogan categories can be defined for any reflection subgroup of a Coxeter group. In the cases where the data does not correspond to a real group, nothing is known about their structure. A unified approach to studying both the Lusztig--Vogan categories corresponding to real groups and their exotic cousins would necessarily be algebraic. This is due to the fact that Lusztig--Vogan categories corresponding to real groups have an equivalent geometric construction in terms of $K$-equivariant constructible sheaves on the flag variety of $G$, but exotic Lusztig--Vogan categories have no such geometric description. 

In this paper, we take the first step toward establishing a unified approach to the study of Lusztig--Vogan categories by giving an algebraic description of the Lusztig--Vogan categories corresponding to the simple real reductive groups of equal rank 2: $\SU(2,1)$, $\Sp(1,1)$, $\Sp_4(\R)$, and split $G_2$. More specifically, we classify indecomposable objects (up to isomorphism and grading shift) and describe the action of generating Soergel bimodules, recovering a version of the $W$-graph\footnote{More precisely, we recover a directed graph with vertices corresponding to the Kazhdan--Lusztig basis of the Lusztig--Vogan module and edges recording the action of the generators of the Hecke algebra. We will refer to this as the $W$-graph. From this graph, one can recover descent set data and reconstruct the $W$-graph in the traditional sense.} of the underlying Lusztig--Vogan module. 

Our goal in these computations is to study the categories directly, without using our knowledge of their Grothendieck groups, in order to develop techniques which work in general. We explore three techniques: (1) a direct approach using results about rings of invariant polynomials, (2) a categorical approach, extracting information about the module category from the monoidal category which acts, and (3) a computational approach, in the form of an algorithm implementable in a computational algebra system. In the paper, we illustrate each set of techniques on a different example. 

We note that the strength of our three techniques is not equal. In particular, the algorithm developed in (3) is not specific to these examples, and works for an arbitrary finite rank Lusztig--Vogan category (up to computational limits). So it provides a much stronger tool in comparison to (1) and (2). 

We dedicate the remainder of the introduction to describing these three sets of techniques in more detail. To do so, we need to define Lusztig--Vogan categories.

\subsection{Lusztig--Vogan categories}
\label{sec: Lusztig--Vogan categories intro}
Assume $K$ is equal rank. Let $W$ be the Weyl group of $G$ and $W_K \subset W$ the Weyl group of $K$. Fix a set of simple reflections $S \subset W$. Denote by ${^KW}$ the set of minimal-length coset representatives for the cosets $W_K \backslash W$ \cite[Cor. 3.4]{Dyer90}. Fix a reflection faithful $\R$-representation $V$ of $W$, and let $R := \op{Sym}(V)$ be the symmetric algebra of $V$, graded so that $\deg V = 2$. This is a polynomial ring in rank$(G)$ variables. Denote by $\SBim$ the monoidal category of Soergel bimodules with respect to the representation $V$ (see \S\ref{sec: Soergel bimodules}). Objects in $\SBim$ are graded $(R,R)$-bimodules, and the monoidal product is the tensor product $\otimes_R$. The category $\SBim$ is additive but not abelian.  

Set $R^K:=R^{W_K}$. For $x \in {^KW}$, denote by $R_x$ the corresponding standard $(R^K, R)$-bimodule (Definition \ref{def: standard bimodules}). The {\em Lusztig--Vogan category} associated to $(G, K)$ is the right $\SBim$-module category generated by standard bimodules:
\[
\mc{N}_{LV}^0 := \langle R_x \otimes_R \SBim \mid x \in {^KW} \rangle_{\oplus, \ominus, (1)}.
\]
Here, the subscript $\oplus, \ominus, (1)$ denotes the additive, graded Karoubi envelope \cite[\S11.2.3]{SBim}. (The superscript $0$ in $\mc{N}_{LV}^0$ indicates that it categorifies the principal block of the Lusztig--Vogan module.) Objects in $\mc{N}_{LV}^0$ are graded $(R^K,R)$-bimodules, and morphisms are degree-preserving bimodule homomorphisms. 

\begin{notation}
    When describing objects in $\mc{N}_{LV}^0$, we often drop the tensor product symbol $\otimes_R$ for brevity, writing $MN$ for $M \otimes_R N$. If we are taking the tensor product over a ring other than $R$, we always include the tensor product with the appropriate subscript.
\end{notation}

Both $\SBim$ and $\mc{N}_{LV}^0$ are {\em Krull--Schmidt} categories, so endomorphism rings of indecomposable objects are local, and objects have unique decompositions into finitely many indecomposable summands. This property motivates our approach to studying them: to ``understand'' a Krull--Schmidt category, one should classify indecomposable objects and describe morphisms between them. In this paper, we tackle the first task by classifying indecomposable objects. A description of morphism spaces in Lusztig--Vogan categories is the topic of ongoing research, and will appear in future work.  

\subsection{A recipe for finding indecomposable objects}
\label{sec: a recipe for finding indecomposable objects}

Our first technique for classifying indecomposable objects in Lusztig--Vogan categories is the most direct. The category of Soergel bimodules is generated (under the operations of $\otimes_R$, $\oplus$, grading shift, and taking direct summands) by the Bott--Samelson bimodules $B_s:=R \otimes_{R^s}R(1)$\footnote{Here $(1)$ is a grading shift.} for simple reflections $s \in S$. This means that Lusztig--Vogan categories can be constructed ``layer by layer'', by acting on the standard generators by $B_s$ for all $s \in S$, decomposing the resulting bimodules into their unique decompositions, then repeating. More precisely, here is a straightforward recipe for finding indecomposable objects in a Lusztig--Vogan category:
\begin{enumerate}
    \item Start with the generators $R_x$ for $x \in {^KW}$. 
    \item Act by $B_s$ for all $s \in S$. 
    \item Decompose the resulting modules into indecomposable summands. 
    \item Decide when two summands are isomorphic. 
    \item Repeat the process on the new indecomposable summands which arise. 
\end{enumerate}
We say that an indecomposable object is in layer $k$ if it is a direct summand of $R_x  B_{s_1}  \cdots  B_{s_k}$ for $x \in {^KW}$ and $(s_1, \ldots, s_k)$ a word in $S$, and is not a direct summand of $R_y  B_{t_1}  \ldots  B_{t_r}$ for any $y \in {^K W}$ and word $(t_1, \ldots, t_r)$  in $S$ with $r < k$. 

Steps (3) and (4) require some work. In \S\ref{sec: four useful lemmas}, we provide a series of lemmas that help us decompose modules, establish isomorphisms, prove that two modules are not isomorphic, and prove that a module is indecomposable. These lemmas, along with knowledge of the ``big indecomposable'' (see \S\ref{sec: the big indecomposable}), are enough to completely determine the indecomposable objects in our first example, $\SU(2,1)$. 

\subsection{The big indecomposable}
\label{sec: the big indecomposable}

There is one indecomposable object that appears in every Lusztig--Vogan category: 
\[
B_{big}:= R^K \otimes_{R^W} R (k),
\]
where $k \in \Z_{>0}$ is a shift depending on the category. We can see that this object lies in $\mc{N}_{LV}^0$ as follows. Because $R$ is always a generator (corresponding to the coset $W_K \in W_K \backslash W$), Lusztig--Vogan categories contain a copy of Soergel bimodules, with a restricted left action:
\[
\mathrm{Res}_{R^K}^R(\SBim) \subset \mc{N}_{LV}^0. 
\]
This means that some indecomposable objects in $\mc{N}_{LV}^0$ can be obtained by taking indecomposable Soergel bimodules\footnote{Indecomposable as $(R,R)$-bimodules.}, restricting the left action to $R^K$, and decomposing further. When we do this to the indecomposable Soergel bimodule
\[
B_{w_0} := R \otimes_{R^W}R(\ell(w_0)),
\]
where $\ell(w_0)$ is the length of the longest element  $w_0 \in W$, we obtain several copies of $B_{big}$, with shifts. This follows from the fact that by the Chevalley--Shephard--Todd theorem, $R$ is a free finite-rank right $R^K$-module. See \S\ref{sec: layer 2 Sp(1,1)} for a precise computation in an example.

The bimodule $B_{big}$ has the property that 
\[
B_{big} B_s \simeq B_{big}(-1) \oplus B_{big}(1),
\]
for any $s \in S$, so no new indecomposable objects can be obtained by applying our recipe in \S\ref{sec: a recipe for finding indecomposable objects} to $B_{big}$. In practice, this means that once our recipe returns $B_{big}$ at a given layer, we can see the light at the end of the tunnel: our process will terminate within the next iteration. 

\begin{remark} \label{rem: the big indecomposable}
    We refer to this as the ``big indecomposable'' because it is the indecomposable object corresponding to the trivial local system on the open $K$-orbit in the flag variety in the Lusztig--Vogan module \cite[Definition 1.5]{LV}, which is the unique orbit of maximal dimension.
\end{remark}

\subsection{Leaning into the monoidal category}
\label{sec: leaning into the monoidal category}
Our second technique for identifying indecomposable objects in Lusztig--Vogan categories is leveraging the known structure of the monoidal category of Soergel bimodules. The category $\SBim$ categorifies the Hecke algebra, and, in this categorification, indecomposable objects correspond to Kazhdan--Lusztig basis elements. In particular, the decompositions of objects in $\SBim$ into indecomposable summands can be computed recursively using Kazhdan--Lusztig polynomials. We can use this structure to our advantage as follows.

If we succeed in decomposing objects of the form $R_x  B_{s_1}  \cdots  B_{s_k}$ in $\mc{N}_{LV}^0$ for small values of $k$ (``lower layers'') using the recipe in \S\ref{sec: a recipe for finding indecomposable objects}, then we can use this to decompose such objects for larger values of $k$ (``higher layers'') in two ways: (1) using known small-$k$ decompositions, and (2) using decompositions of Soergel bimodules. By the Krull--Schmidt property of the category, any indecomposable object occurring in one decomposition must also occur in every other. In some cases, this allows us to deduce that objects have certain indecomposable summands without actually writing down splitting morphisms. In our second example, $\Sp(1,1)$, we illustrate this strategy. 

\subsection{Computational techniques}
\label{sec: computational techniques}

While the examples of $\SU(2,1)$ and $\Sp(1,1)$ were tractable by hand, we found that our techniques were insufficient to describe the Lusztig--Vogan category corresponding to split $G_2$. To this end, we implemented an algorithm to iteratively compute the indecomposable objects using a computer. This algorithm carries out the recipe in \S\ref{sec: a recipe for finding indecomposable objects} in the natural way, starting with the generating objects (``layer 0''), and keeping track of the new indecomposables which appear at each subsequent layer. The decomposition step (step (3) in \S\ref{sec: a recipe for finding indecomposable objects}) is accomplished by computing the primitive idempotents in the corresponding endomorphism rings, then appropriately encoding their image. We explain this and give a more detailed description of the algorithm in \S\ref{sec: an algorithm to compute indecomposable objects}. The implementation of this algorithm can be found at \cite{magma-code}.

\subsection{Structure of paper}
\label{sec: structure of paper}
The paper is organised as follows. In \S\ref{sec: background}, we establish our conventions on Soergel bimodules and define Lusztig--Vogan categories. In \S\ref{sec: four useful lemmas}, we prove the lemmas which allow us to implement the recipe in \ref{sec: a recipe for finding indecomposable objects}. Then in \S\ref{sec: A type A Lusztig--Vogan category} - \S\ref{sec: A type G Lusztig--Vogan category}, we describe our three main Lusztig--Vogan categories, starting each section with a description of the Lie theoretic objects (the complex group $G$, the involution $\theta$, and the fixed point group $K$), then computing the indecomposables layer by layer. In our first example, the group $\SU(2,1)$, we do this using only the recipe in \ref{sec: a recipe for finding indecomposable objects} and knowledge of the big indecomposable. In our second example, $\Sp(1,1)$, we utilise the monoidal category technique described in \S\ref{sec: leaning into the monoidal category}. In our third example, the split real form of $G_2$, we show that our manual techniques are insufficient, justifying the need for the algorithm described in \S\ref{sec: computational techniques}. The final section \S\ref{sec: an algorithm to compute indecomposable objects} is devoted to a description of the algorithm. We conclude with the $W$-graphs of the Lusztig--Vogan categories corresponding to the final two examples: split $G_2$ and $\Sp_4(\R)$.

\subsection{Acknowledgements}
\label{sec: Acknowledgements}

We thank Geordie Williamson for many helpful conversations, and for suggesting that we tackle $G_2$ using Magma. 
The first and third authors were supported by the Commonwealth through Australian Government Research Training Program Scholarships (\href{https://doi.org/10.82133/C42F-K220}{\texttt{doi:10.82133/C42F-K220}}). The second author thanks the Max Planck Institute for Mathematics in Bonn for providing excellent working conditions. ChatGPT 5.6 was used for proofreading and figure formatting. All writing and mathematical arguments were done by the authors. 

\section{Background}
\label{sec: background}

\subsection{Graded modules}
\label{sec: Graded modules} 
All rings and modules in this paper are graded. For polynomial rings $R$ and $T$ over $\R$, graded so that $T^0=\R=R^0$, denote by $(T,R)$-gBim the category of finitely generated\footnote{As both left and right modules.} $\Z$-graded $(T,R)$-bimodules such that left and right action of $\R$ agree. Morphisms in $(T,R)$-gBim are homogeneous degree 0 bimodule homomorphisms. This category has a shift functor given by $M(i)^j:=M^{i+j}$. To keep track of graded ranks, we introduce the following notation: for a Laurent polynomial $p = \sum p_i v^i \in \Z_{\geq 0}[v^{\pm 1}]$, set 
\begin{equation}
    \label{eq: graded degree}
    M^{\oplus p} := \bigoplus_{i \in \Z} M(i)^{\oplus p_i}.
\end{equation}

\subsection{Soergel bimodules} 
\label{sec: Soergel bimodules}
We return to the set-up in \S\ref{sec: Lusztig--Vogan categories intro}: The Weyl group $W$ of $G$, along with a fixed set of simple roots $S$, forms a Coxeter system $(W,S)$. Fix a reflection faithful realisation $(V, \{\alpha_s\}_{s\in S}, \{ \alpha_s^\vee\}_{s \in S})$ of $(W,S)$ over $\R$ \cite[Def. 5.33, Def. 5.37]{SBim}\footnote{In all of our hand computations, we use the root datum realisation \cite[Ex. 5.41]{SBim}. The magma implementation of our algorithm uses the geometric realisation \cite[Ex. 5.40]{SBim}.}. We consider the symmetric algebra $R:=\op{Sym}(V)$, graded so that $\deg V = 2$. We identify $R$ with the polynomial ring $\R[\alpha_s \mid s \in S]$, where $\deg \alpha_s = 2$.

For an expression $\underline{w}=(s_1, \ldots, s_n)$ of an element $w \in W$, denote by 
\begin{equation}
    \label{eq: Bott-Samelson}
    BS(\underline{w}):=B_{s_1} \otimes_R \cdots \otimes_R B_{s_n} \simeq R\otimes_{R^{s_1}} \cdots \otimes_{R^{s_n}} R (n)
\end{equation}
the corresponding Bott--Samelson bimodule in $(R, R)$-gBim, where $B_{s_i}:= R \otimes_{R^{s_i}} R(1)$. The category of Soergel bimodules, denoted $\SBim$, is the additive Karoubian subcategory of $(R,R)$-gBim generated by Bott--Samelson bimodules and their shifts. It is a monoidal category under the operation $\otimes_R$. 

Indecomposable objects in $\SBim$ are parameterised by $w \in W$. Denote by $B_w$ the indecomposable object corresponding to $w \in W$. Then $B_w$ appears with multiplicity $1$ as a direct summand in $BS(\underline{w})$ for any reduced expression $\underline{w}$ of $w$. Since $\SBim$ is a monoidal category with a shift functor, the (additive) Grothendieck group of $\SBim$ has the structure of a $\Z[v^{\pm 1}]$-algebra. This algebra is isomorphic to the Hecke algebra of the Coxeter system $(W,S)$.

\subsection{Lusztig--Vogan categories}
\label{sec: Lusztig--Vogan categories}
In this paper, we will impose the assumption that the real group $G_\R$ is of {\em equal rank}, meaning that $\mathrm{rank}(G) = \mathrm{rank}(K)$. Lusztig--Vogan categories can be defined for groups which are not of equal rank (see \cite[\S 6]{LarsonRomanov} for a general construction), but imposing an equal rank assumption simplifies their definition. 

As in \S\ref{sec: Lusztig--Vogan categories intro}, denote by $W_K$ the Weyl group of $K$. There is a natural embedding $W_K \hookrightarrow W$ \cite[Lem. 6.1.3]{LarsonRomanov}, so we can realise $W_K$ as a reflection subgroup of $W$. The cosets $W_K \backslash W$ have unique minimal length representatives by \cite[Cor. 3.4]{Dyer90}; denote this set by ${^K W}$. 

\begin{remark}
\label{rem: conjugate real forms}
    (Conjugate real forms correspond to conjugate reflection subgroups) We follow the conventions in \cite{algorithms} with regard to real forms: to us, a {\em real form}  of a complex reductive algebraic group $G$ is a $G$-conjugacy class of holomorphic involutions on $G$. (Note that to actually obtain a real Lie group from a holomorphic involution, one must first compose it with the unique compact involution and then take the fixed point subgroup.) Two holomorphic $G$-involutions $\theta$ and $\theta'$ are conjugate if and only if the corresponding reflection subgroups $W_K$ and $W_{K'}$ are conjugate in $W$. Hence conjugate reflection subgroups will return equivalent Lusztig--Vogan categories. It turns out that any two reflection subgroups which are related via a Coxeter diagram automorphism also return equivalent categories---this situation arises in Example \ref{sec: A type C Lusztig--Vogan category}.  
\end{remark}

Set $R^K:=R^{W_K}$. The generators of the Lusztig--Vogan category are the standard bimodules corresponding to elements in $^KW$. 
\begin{definition}
    \label{def: standard bimodules}
    For $x \in W$, define the {\em standard $(R^K, R)$-bimodule} $R_x$ to be the vector space $R$, considered as a left $R^K$-module with the multiplication action, and as a right $R$-module with action 
    \[
    r \cdot_x g:= r x(g)
    \]
    for $r \in R_x$ and $g \in R$. 
\end{definition}

\begin{remark}
    \label{rem: R^K vs R in standard bimodules}
    One can define standard $(R,R)$-bimodules in the same way, as in \cite[\S5.2]{SBim}. In \S\ref{sec: a tool for decomposing bimodules}, we consider standard bimodules in $(R,R)$-gBim, and use the same symbol $R_x$. Throughout this paper, the ring which is acting on the left should be clear from context. 
\end{remark}

Two standard $(R^K,R)$-bimodules $R_x$ and $R_y$ are isomorphic if and only if $x$ and $y$ are in the same right $W_K$-coset \cite[Lem. 6.2.3]{LarsonRomanov}, so the isomorphism classes of standard bimodules are in bijection with $W_K \backslash W$. 

\begin{definition}
    \label{def: Lusztig--Vogan category} The {\em Lusztig--Vogan category associated to the pair $(G,K)$} is the full additive Karoubian subcategory of $(R^K,R)$-gBim generated by the standard $(R^K,R)$-bimodules under the right tensor action of Soergel bimodules:
    \[
    \mc{N}_{LV}^0(G,K):= \langle R_x \otimes_R \SBim \mid x \in {^K W} \rangle_{\oplus, \ominus, (1)}.
    \]
    Here the subscripts $\oplus$, $\ominus$, and $(1)$ indicate closure under the operations of direct sum, taking direct summands, and grading shift, respectively. 
\end{definition}
By construction, a Lusztig--Vogan category is a module category over the monoidal category of Soergel bimodules. Hence its (additive) Grothendieck group is a module over the corresponding Hecke algebra. This module is a submodule of the Lusztig--Vogan module, whose definition can be found in \cite[Lem. 3.5]{LV} or \cite[Def. 6.4]{Vogan3}. We will not repeat it here. 

\begin{theorem}
    (Combine \cite[Thm. 4.1.3]{LarsonRomanov} and \cite[Thm. 3.2]{BV}) The category $\mc{N}_{LV}^0(G,K)$ categorifies the principal block of the Lusztig--Vogan module corresponding to $(G,K)$. 
\end{theorem}

\section{Four useful lemmas}
\label{sec: four useful lemmas}

Now we start to build our toolbox. To implement the recipe in \S\ref{sec: a recipe for finding indecomposable objects} for finding indecomposable objects in a category, we must be able to do four things: 
\begin{enumerate}
    \item Establish isomorphisms between objects. 
    \item Decompose objects into direct sums of sub-objects.
    \item Prove that certain objects are not isomorphic.
    \item Prove that objects are indecomposable. 
\end{enumerate}
Each of these tasks is difficult. In this section, we provide four lemmas which help us in certain circumstances. The results in this section hold for an arbitrary Coxeter system $(W,S)$ and reflection subgroup $W_K \subset W$. 

\subsection{A tool for finding isomorphisms}
\label{sec: a tool for finding isomorphisms}

Objects in $\mc{N}_{LV}^0$ are shifts of summands of $(R^K,R)$-bimodules of the form 
\[
R_xB_{s_1} \cdots B_{s_n}
\]
for sequences of simple reflections $(s_1, \ldots, s_n)$ and $x \in W$. For various choices of $x$ and  $s_i$, these bimodules will be isomorphic. In particular, for $n=1$, we have the following lemma. 
\begin{lemma}
    \label{lem: isomorphism lemma} (Isomorphism Lemma) For any $s \in S$, $x \in W$, 
    \[
    R_{xs}B_s \simeq R_x B_s
    \]
    as graded $(R,R)$-bimodules.
\end{lemma}
\begin{proof}
    For $s \in S$, the identity map 
    \begin{equation}
        \label{eq: isomorphism lemma proof}
            R_s \otimes_{R^s} R (1) \rightarrow R \otimes_{R^s} R(1); f \otimes g \mapsto f \otimes g 
    \end{equation}
    is a well-defined isomorphism of graded $(R,R)$-bimodules, since for $g \in R^s$, 
    \[
    f \otimes g = f \cdot_s g \otimes 1 = f s(g) \otimes 1 = fg \otimes 1. 
    \]
    The result then follows from taking the tensor product of \eqref{eq: isomorphism lemma proof} with $R_x$ and using the isomorphism 
    \begin{equation}
        \label{eq: R_xR_s is R_{xs}}
        R_x R_s \xlongrightarrow{\simeq} R_{xs}; f \otimes g \mapsto fx(g).
    \end{equation}
\end{proof}
By restriction, Lemma \ref{lem: isomorphism lemma} also holds for $(R^K, R)$-bimodules, which is how we use it throughout the paper. 

\subsection{A tool for decomposing bimodules}
\label{sec: a tool for decomposing bimodules}
For any $s \in S$, there are two important degree $1$ homogeneous elements in $B_s$:
\begin{equation}
\label{eq: c_s and d_s}
c_s:= \frac{1}{2}(\alpha_s \otimes 1 + 1 \otimes \alpha_s) \text{ and } d_s:= \frac{1}{2} (\alpha_s \otimes 1 - 1 \otimes \alpha_s).
\end{equation}
The element $c_s$ has the property that the left and right $R$-action agree, and $d_s$ has the property that the right $R$-action differs from the left $R$-action by $s$ \cite[Exercise 5.6]{SBim}. The existence of $d_s$ provides a short exact sequence 
\begin{equation}
    \label{eq: initial ses}
    0 \longrightarrow R_s(-1) \xrightarrow{1 \mapsto d_s} B_s \xrightarrow{\hspace{1mm}\mu_{\mathrm{id}}\hspace{1mm}} R(1) \longrightarrow 0,
\end{equation}
in $(R,R)$-gBim, where $\mu_{\mathrm{id}}(f \otimes g) = fg$. 

For any $x \in W$, $R_x$ is free as a right $R$-module, so the functor $R_x \otimes_R - $ is exact. Applying this functor to \eqref{eq: initial ses} and using the inverse isomorphism to \eqref{eq: R_xR_s is R_{xs}}, we obtain for any $x \in W$ and $s \in S$ the following short exact sequence in $(R,R)$-gBim:
\begin{equation}
    \label{eq: the short exact sequence}
    0 \longrightarrow R_{xs}(-1) \longrightarrow R_x B_s \longrightarrow R_x(1) \longrightarrow 0. 
\end{equation}

In the category $(R,R)$-gBim, the sequence \eqref{eq: the short exact sequence} never splits. However, in $(R^K, R)$-gBim, \eqref{eq: the short exact sequence} sometimes splits. Examining this splitting provides a tool for decomposing ``layer 1'' objects in $\mc{N}_{LV}^0$; i.e., those of the form $R_x B_s$. 

\begin{lemma}
    \label{lem: splitting lemma} (Splitting Lemma) The short exact sequence \eqref{eq: the short exact sequence} splits in $(R^K, R)$-gBim if and only if $x^{-1}(R^K) \subseteq R^s$. 
\end{lemma}
\begin{proof}
Define
\[
\psi: R_x(1) \rightarrow R_xB_s; 1 \mapsto 1 \otimes 1.
\]
Here we are identifying $R_xB_s$ with $R_x \otimes_{R^s} R (1)$ under the obvious isomorphism. This map is a well-defined morphism of graded $(R^K,R)$-bimodules if and only if, for every $f \in R^K$, the left action of $f$ on $1 \otimes 1$ agrees with the right action of $x^{-1}(f)$ (since the generator $1 \in R_x$ satisfies this and $\psi(1)$ determines $\psi$). The following computation shows that this happens when $x^{-1}(f) \in R^s$:
\[
f\cdot 1 \otimes 1 = f \otimes 1= 1 \cdot_x x^{-1}(f) \otimes 1 = 1 \otimes x^{-1}(f) = 1 \otimes 1 \cdot x^{-1}(f).
\]
Moreover, if $f\cdot 1 \otimes 1 = 1 \otimes 1 \cdot x^{-1}(f)$ for $f \in R^K$, then $x^{-1}(f) \in R^s$. To see this, we can first write $x^{-1}(f) = a + b\alpha_s$, for $a, b \in R^s$ by \cite[Example 4.12]{SBim}. Then,
\[
0 = f \otimes_{R^s} 1 - 1 \otimes_{R^s} x^{-1}(f) = (x(\alpha_s) \otimes_{R^s} 1 - 1 \otimes_{R^s} \alpha_s)b.
\]
Since $\{1 \otimes_{R^s} 1,\ x(\alpha_s) \otimes_{R^s} 1 - 1 \otimes_{R^s} \alpha_s\}$ forms a basis for $R_xB_s$ as a right $R$-module, we know that $x(\alpha_s) \otimes_{R^s} 1 - 1 \otimes_{R^s} \alpha_s$ cannot be zero. Thus the only possibility is that $b = 0$ and hence $x^{-1}(f) = a \in R^s$.

We conclude that $\psi$ is a morphism in $(R^K,R)$-gBim. It clearly splits the sequence \eqref{eq: the short exact sequence}.
\end{proof}

\subsection{A tool for establishing non-isomorphism}
\label{sec: a tool for establishing non-isomorphism} 

The short exact sequence \eqref{eq: initial ses} has a dual version provided by the element $c_s$ in \eqref{eq: c_s and d_s}:
\begin{equation}
    \label{eq: nabla filtration} 
    0 \longrightarrow R(-1) \xrightarrow{1 \mapsto c_s} B_s \xrightarrow{\hspace{1mm} \mu_s \hspace{1mm}} R_s(1) \longrightarrow 0,
\end{equation}
where $\mu_s(f \otimes g) = f \cdot s(g)$. By taking repeated tensor products of the sequences \eqref{eq: initial ses} and \eqref{eq: nabla filtration} for the simple reflections in an expression $\underline{w}$ of $w \in W$, one can construct a filtration of the object $R_x BS(\underline{w})$ in $\mc{N}_{LV}^0$, whose subquotients are shifts of standard $(R^K, R)$-bimodules. 

There are choices to be made in this construction: for each simple reflection $s$ in $\underline{w}$, one chooses which sequence, \eqref{eq: initial ses} or \eqref{eq: nabla filtration}, to tensor with. Because the shifts in \eqref{eq: initial ses} and \eqref{eq: nabla filtration} are different, different choices will result in different grading shifts in the standard subquotients of the final filtration. However, the grading change is the only discrepancy, so, \textit{up to grading}, standard subquotients appearing in any filtration of $R_x BS(\underline{w})$ obtained by tensoring \eqref{eq: initial ses} and \eqref{eq: nabla filtration} will be uniquely determined by $x$ and $\underline{w}$. 

Any such filtration restricts to a filtration of a direct summand of $R_xBS(\underline{w})$. Because standard bimodules are indecomposable, the subquotients in the restricted filtration are also (shifts of) standard bimodules \cite[p.320]{EW}, and their (ungraded) multiplicity does not depend on the choices of \eqref{eq: initial ses} and \eqref{eq: nabla filtration} in the construction of the filtration of $R_xBS(\underline{w})$. In particular, every object in $\mc{N}_{LV}^0$ has a well-defined multiset of standard factors\footnote{This is exactly analogous to the situation for Soergel bimodules, as explained in \cite[\S3.4]{EW}. Note that for our purposes, we do not need the full theory of standard filtrations or characters, so we do not develop it here.}.

\begin{definition}
    \label{def: standard factors} Let $M$ be an object in $\mc{N}_{LV}^0$. The set $\{R_x\}$ of standard $(R^K, R)$-bimodules with the property that a shift of $R_x$ appears as a subquotient in a filtration of $M$ constructed in the method described above (counted with multiplicity) is the multiset of (ungraded) {\em standard factors} of $M$.
\end{definition}

By the discussion above, the standard factors of $M$ are well-defined. In particular, this gives us a criterion for determining when two modules are not isomorphic. 

\begin{lemma}
    \label{lem: standard factors} 
    Let $M$ and $N$ be objects in $\mc{N}_{LV}^0$. If $M$ and $N$ do not share the same set of standard factors, then they are not isomorphic. 
\end{lemma}

\subsection{A tool for establishing indecomposability}
\label{sec: tools for establishing indecomposability}

Lemmas \ref{lem: isomorphism lemma}, \ref{lem: splitting lemma} and \ref{lem: standard factors} give us tools to decompose bimodules and establish when they are and are not isomorphic. However, to identify indecomposable objects in our categories, we need tools to establish indecomposability. The simplest way of doing this is with the following lemma. 
\begin{lemma} 
\label{lem: indecomposability via generation by homogeneous elt}
     Any object in $(R^K,R)$-gBim which is generated by a single homogeneous element is indecomposable.
\end{lemma}
\begin{proof}
    This is a mild adaptation of \cite[Lem. 4.34]{SBim}. The proof given there also applies in this setting. 
\end{proof}
The converse of Lemma \ref{lem: indecomposability via generation by homogeneous elt} is not true: not every indecomposable $(R^K,R)$-bimodule is generated by a single homogeneous element. Unfortunately, this situation arises frequently in Lusztig--Vogan categories. Though our indecomposable bimodules are often not cyclically generated by a homogeneous element, some of them are ``almost cyclically generated'', in that there exists a large cyclically generated submodule with the property that every element of the module can be mapped into the submodule via only the right $R$-action. It turns out that this condition is sufficient to establish indecomposability. More precisely, the following (much less slick) indecomposability lemma suits our purposes. 
\begin{lemma}
    \label{lem: indecomposability lemma} (Indecomposability Lemma) Assume $M \in (R^K, R)$-gBim satisfies the following four conditions: 
    \begin{enumerate}
        \item $M$ is positively graded; 
        \item The degree zero component $M^0 \simeq \R$;
        \item $M$ is graded free of finite rank as a right $R$-module; 
        \item There exists $m \in M^0$ such that for all $n \in M$, there exists a nonzero element $f \in R$ such that $n \cdot f \in R^K \cdot m \cdot R$. 
    \end{enumerate}
    Then $M$ is indecomposable as a graded $(R^K, R)-$bimodule. 
\end{lemma}
\begin{proof}
    We will establish indecomposability by proving that the endomorphism ring of $M$ is local\footnote{Recall from \S\ref{sec: Lusztig--Vogan categories} that Lusztig--Vogan categories are {\em Krull--Schmidt}, so objects are indecomposable if and only if they have local endomorphism rings.}. Let $\phi: M \rightarrow M$ be a graded $(R^K,R)$-bimodule morphism. Set 
    \[
    M':= R^K \cdot m \cdot R,
    \]
    where $m \in M$ is the element guaranteed by (4). Since $m \in M^0 \simeq \R$, we have $\phi(m) = cm$ for some $c \in \R$. For $(r, s) \in (R^K, R)$, 
    \[
    \phi(rms) = r \phi(m)s = crms,
    \]
    so $\phi|_{M'} = c \hspace{1mm} \mathrm{id}$. Consider $n \in M - M'$. By (4), there exists a nonzero $f \in R$ such that $n \cdot f \in M'$; i.e., $nf=rms$ for some $(r, s) \in (R^K, R)$. Then, 
    \[
    \phi(n)f = \phi(nf) = \phi(rms) = crms = cnf,
    \]
    so $(\phi(n) - cn) f = 0$. Because $M$ is free as a right $R$-module, it has no torsion elements. Hence $\phi(n)=cn$. We conclude that $\End_{(R^K,R)}(M) \simeq \R$ is local. 
\end{proof}

\section{A type $A$ Lusztig--Vogan category}
\label{sec: A type A Lusztig--Vogan category}

Our first example corresponds to the group $\SU(2,1)$. In this example, $W_K$ is generated by a single simple reflection. We use this example to illustrate that the direct approach in \S\ref{sec: a recipe for finding indecomposable objects} is possible in small examples. In this section, we classify indecomposable objects in the Lusztig--Vogan category for $SU(2,1)$ by using the lemmas in \S\ref{sec: four useful lemmas} and doing direct computations with polynomial rings. Such computations quickly become intractable in larger examples such as $G_2$. 

\subsection{The group $SU(2,1)$}
\label{sec: the group SU(2,1)}

Let $G=\SL_3(\C)$. Define a Cartan involution $\theta:G \rightarrow G$ on $g \in G$ by 
\[
\theta: g \mapsto \bp 1 & 0 & 0 \\ 0 & 1 & 0 \\ 0 & 0 & -1 \ep g \bp 1 & 0 & 0 \\ 0 & 1 & 0 \\ 0 & 0 & -1 \ep.
\]
The elements fixed by this involution form the group 
\[
K:= G^\theta = \left\{ \bp a & b & 0 \\ c & d & 0 \\ 0 & 0 & e \ep \middle| \bp a & b \\ c & d \ep  \in \GL_2(\C), e = \left(\det\bp a & b \\ c & d \ep \right) ^{-1} \right\} \simeq \GL_2(\C).
\]
The composition of $\theta$ with the compact involution $\sigma_c$ 
of $G$, which sends an element to its conjugate transpose inverse, results in an antiholomorphic involution $\sigma:= \theta \circ \sigma_c$, whose fixed points are the real group $\SU(2,1)$. The groups $G$ and $K$ share a diagonal maximal torus $T$, so the group $\SU(2,1)$ is of equal rank. 

The Weyl group $W$ of $G$ is the symmetric group on $3$ letters, which we will present in its Coxeter presentation:
\[
W=\langle s, t \mid s^2=t^2=1, (st)^3=1 \rangle.
\]
The Weyl group $W_K$ of $K$ is naturally identified\footnote{Under the embedding $W_K \hookrightarrow W$ given in \cite[Lem. 6.1.3]{LarsonRomanov}.} with the subgroup 
\[
W_K = \langle s \rangle \subset W. 
\]

Hence the quotient $W_K \backslash W$ consists of three cosets, with minimal coset representatives 
\begin{equation}
    \label{eq: min length SU(2,1)}
    {}^K W = \{1, t, ts\}. 
\end{equation}

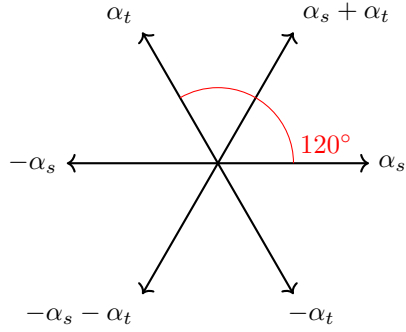
\begin{figure}
  \begin{tikzpicture}
    \foreach \i in {0, 1, ..., 5} {
      \draw[thick, ->] (0, 0) -- (\i*60:2);
    }
    \node[right] at (2, 0) {$\alpha_s$};
    \node[right] at (1, 2) {$\alpha_s+\alpha_t$};
    \node[above left] at (4*30:2) {$\alpha_t$};
    \node[left] at (-2, 0) {$-\alpha_s$};
    \node[left] at (-1, -2) {$-\alpha_s-\alpha_t$};
    \node[right] at (0.8,-2) {$-\alpha_t$};
    \draw[thin, red] (1, 0) arc[radius=1, start angle=0, end angle=4*30];
    \node[right, red] at (15:1) {$120^\circ$};
  \end{tikzpicture}
  \caption{The root system of type $A_2$. Dotted lines are reflecting hyperplanes and solid lines are root vectors. }
  \label{fig: A2 root system}
\end{figure}

The root system associated to $G$ is of type $A_2$ (Figure \ref{fig: A2 root system}). Denote by $V$ the $\R$-span of the simple roots $\alpha_s, \alpha_t$. The $\R$-vector space $V$ is a faithful representation of $W$, with $s$ (resp. $t$) acting by reflection across the hyperplane orthogonal to $\alpha_s$ (resp. $\alpha_t$). 

Set $R=\op{Sym}(V)$ to be the symmetric algebra of $V$, graded with $\op{deg} V=2$. We identify $R$ with the polynomial ring 
\begin{equation}
    \label{eqn: R for SU(2,1)}
    R=\R[\alpha_s, \alpha_t].
\end{equation}
By the Chevalley--Shephard--Todd theorem (\cite{Chevalley, ShephardTodd}, see also \cite[\S3.5]{HumphreysCoxeter}), $R^K$ is a polynomial ring with homogeneous generators $\alpha_s^2, \alpha_t + \frac{1}{2} \alpha_s$ (see also \cite[Example 4.12]{SBim}):
\begin{equation}
\label{eq: R^K for SU(2,1)}
R^K = \R[\alpha_s^2, \alpha_t + \frac{1}{2} \alpha_s].
\end{equation}

\subsection{The category $\mc{N}_{LV}^0(\SU(2,1))$}
\label{sec: the category N_LV for SU(2,1)}
We now proceed to the study of the Lusztig--Vogan category associated to $\SU(2,1)$.

\subsubsection{Generators}
\label{sec: generators Sp(1,1)}
The minimal coset representatives $1$, $t$ and $ts$ of $W_K \backslash W$ \eqref{eq: min length SU(2,1)} each contribute a generator to $\mc{N}^0_{LV}$: 
\begin{equation}
    \label{eq: generators SU(2,1)}
    \mc{N}_{LV}^0 =  \langle \left\{ R, R_t, R_{ts}  \right\} \otimes_R \SBim  \rangle _{\oplus, \ominus, (1)} \subset (R^K,R)\textrm{-gBim}.
\end{equation}

Standard bimodules are always indecomposable, as they are generated by $1$  (Lemma \ref{lem: indecomposability via generation by homogeneous elt}), so our generating set provides three indecomposable objects: 
\begin{equation}
    \label{eq: layer 0 indecomposables SU(2,1)}
    R, R_t, R_{ts}.
\end{equation}

To find the remaining indecomposable objects, we proceed following the recipe in \S\ref{sec: a recipe for finding indecomposable objects}, using the tools in \S\ref{sec: four useful lemmas}. The first layer is obtained by tensoring our generators with $B_s$ and $B_t$.

\subsubsection{Layer 1}
\label{sec: layer 1 SU(2,1)}

Action by $B_s$ and $B_t$ on our three generators results in six objects in our category:
\begin{equation}
    \label{eq: full layer 1 list SU(2,1)}
    RB_s, \hspace{2mm} RB_t, \hspace{2mm} R_t B_s, \hspace{2mm}  R_t B_t, \hspace{2mm} R_{ts}B_s,  \text{ and } R_{ts}B_t. 
\end{equation}
\begin{notation}
\label{rem: notation - R on the left}
    To clarify which ring is acting on the left, we will always include a factor of $R$ on the left when we are considering objects in $\SBim$ as $(R^K, R)$-bimodules. 
\end{notation}
An application of the isomorphism Lemma (Lemma \ref{lem: isomorphism lemma}) provides two isomorphisms between objects on the list \eqref{eq: full layer 1 list SU(2,1)}: 
\[
RB_t \simeq R_t B_t \text{ and } R_t B_s \simeq R_{ts} B_s. 
\]
The remaining objects all have distinct standard factors, so they are not isomorphic by Lemma \ref{lem: standard factors}. Hence up to isomorphism, our first layer contributes four ``Bott--Samelson-type'' objects in our category:
\begin{equation}
    \label{eq: bott-samelsons for layer 1 of SU(2,1)}
    RB_s, \hspace{2mm} RB_t, \hspace{2mm} R_t B_s, \text{ and } R_{ts}B_t.
\end{equation}

Next we determine which of these Bott--Samelson objects decompose. An easy computation using \eqref{eq: R^K for SU(2,1)} shows that 
\[
(ts)^{-1}(R^K) = \R[\alpha_t^2, \alpha_s + \frac{1}{2} \alpha_t] = R^t.
\]
This indicates that we can decompose two of the objects in \eqref{eq: bott-samelsons for layer 1 of SU(2,1)} using the splitting lemma (Lemma \ref{lem: splitting lemma}):
\[
RB_s \simeq R^{\oplus v^{-1} + v} \text{ and } R_{ts} B_t \simeq R_{ts}^{\oplus v^{-1} + v}.
\]

Finally, we must establish that the remaining two objects in \eqref{eq: bott-samelsons for layer 1 of SU(2,1)} are indecomposable. This can be done using Lemma \ref{lem: indecomposability via generation by homogeneous elt}. Indeed, one can check that:
\begin{enumerate}
    \item the subrings $R^s$ and $R^t$ generate\footnote{This holds for general Coxeter systems, as long as $s \neq t$ and $m_{st} \neq \infty$, see \cite[Example 4.37]{SBim}.} $R$, and 
    \item the subrings $t(R^s)$ and $R^s$ generate $R$.
\end{enumerate}
These two facts imply that $RB_t$ and $R_t B_s$ are each generated by $1 \otimes 1$ as $(R^K, R)$-bimodules.

We conclude that from our first iteration of the recipe in \S \ref{sec: a recipe for finding indecomposable objects}, we have obtained two new layer 1 indecomposable objects:
\begin{equation}
    \label{eq: layer 1 indecomposables SU(2,1)}
    RB_t \text{ and } R_t B_s.
\end{equation}

\subsubsection{Layer 2}
\label{sec: layer 2 SU(2,1)}

Now we iterate the process. Action on the layer 1 indecomposable objects in \eqref{eq: layer 1 indecomposables SU(2,1)} by $B_s$ and $B_t$ yields four objects
\begin{equation}
    \label{eq: layer 2 full list SU(2,1)}
    RB_t B_s, \hspace{2mm} RB_t B_t, \hspace{2mm} R_tB_sB_t, \text{ and } R_t B_s B_s.
\end{equation}

We can decompose two objects on this list using the $(R,R)$-bimodule decomposition $B_r B_r \simeq B_r^{\oplus v^{-1} + v}$ for $r \in S$ \cite[Example 4.35]{SBim}: 
\[
RB_t B_t \simeq (RB_t)^{\oplus v^{-1} + v} \text{ and } R_t B_s B_s \simeq (R_t B_s)^{\oplus v^{-1} + v}.
\]

We claim that the remaining two objects decompose, with the big indecomposable as a factor. We will establish these decompositions directly by writing down split short exact sequences. 

Define the $(R^K,R)$-bimodule 
\begin{equation}
    \label{eq: big indecomposable SU(2,1)}
    B_{big}:= R^K \otimes_{R^W} R (2). 
\end{equation}
Clearly $B_{big}$ is generated by $1 \otimes 1$, and is thus indecomposable by Lemma \ref{lem: indecomposability via generation by homogeneous elt}. It appears in the third layer of our category by the argument in \S\ref{sec: the big indecomposable}. (In this example, the length of $w_0$ is $3$.) As it happens, it also appears in layer 2.  

Consider the sequence of graded $(R^K, R)$-bimodules 
\begin{equation}
    \label{eq: B_big short exact sequence SU(2,1)}
    0 \longrightarrow B_{big} \xlongrightarrow{\iota} RB_t B_s \xlongrightarrow{\psi} R \longrightarrow 0
\end{equation}
where $\iota (1 \otimes 1) = 1 \otimes 1 \otimes 1$, and 
\[
\psi(f \otimes g \otimes h) = \partial_s(fg) h, 
\]
where $\partial_s: R \rightarrow R_s(-2)$ is the Demazure operator 
\begin{equation}
    \label{eq: demazure operator}
    \partial_s(f) := \frac{f - s(f)}{\alpha_s} .
\end{equation}
Here we are identifying $RB_tB_s$ with $R \otimes_{R^t} R \otimes_{R^s} R(2)$. One can check that the sequence \eqref{eq: B_big short exact sequence SU(2,1)} is exact\footnote{We will not give full details of the argument here because it is a little tedious, but this can be done using the strategy in \cite[Exercise 4.41]{SBim}. From \eqref{eq: R idempotent in SU(2,1)}, one can obtain the idempotent corresponding to the summand of $R$, then use this to obtain the idempotent corresponding to $B_{big}$. This provides the projection map $RB_tB_s \rightarrow B_{big}$, which guarantees that the arrow $\xlongrightarrow{\iota}$ in \eqref{eq: B_big short exact sequence SU(2,1)} is injective. Surjectivity of $\xlongrightarrow{\psi}$ is clear. One can then check exactness in the middle term by doing so on the level of right $R$-modules.}. 

The $(R^K, R)$-module morphism 
\begin{equation}
    \label{eq: R idempotent in SU(2,1)}
    \psi': R \rightarrow RB_t B_s; f \mapsto -c_t \otimes f
\end{equation}
provides a splitting of \eqref{eq: B_big short exact sequence SU(2,1)}, where $c_t \in B_t$ is the element in \eqref{eq: c_s and d_s}. This establishes the decomposition 
\[
RB_tB_s  = B_{big} \oplus R. 
\]

Similarly, we have an exact sequence 
\begin{equation}
    \label{eq: B_big in SU(2,1) via the other Bott-Samelson}
    0 \longrightarrow B_{big} \xlongrightarrow{\iota} R_tB_sB_t \xlongrightarrow{\phi} R_{ts} \longrightarrow 0
\end{equation}
given by $\iota(1 \otimes 1) = 1 \otimes 1 \otimes 1 \otimes 1$ and 
\[
\phi(f \otimes g \otimes h \otimes k) = ts(\partial_t(st(f) s(g) h) k),
\]
where $1 \otimes 1 \otimes 1 \otimes 1$ and $f \otimes g \otimes h \otimes k$ are viewed as simple tensors in $R_t \otimes_R R \otimes_{R^s} R \otimes_{R^t} R (2) \simeq R_t B_s B_t$, and $\partial_t$ is the Demazure operator corresponding to $t \in S$ \eqref{eq: demazure operator}. 

The composition of the $(R^K, R)$-bimodule morphism 
\begin{equation}
    \label{eq: R_t idempotent in SU(2,1)}
    R_{t}R_s \xrightarrow{\phi'} R_t B_s B_t; f \otimes g \mapsto 1 \otimes d_s \otimes st(f)s(g)
\end{equation}
with the isomorphism $R_{ts}\xlongrightarrow{\simeq} R_t R_s$ whose inverse is given in \eqref{eq: R_xR_s is R_{xs}} provides a splitting for \eqref{eq: B_big in SU(2,1) via the other Bott-Samelson}. Here $d_s \in B_s$ is the element in \eqref{eq: c_s and d_s}. This establishes the decomposition 
\[
R_tB_sB_t = B_{big} \oplus R_{ts}.
\]
\begin{remark}
    \label{rem: diagrammatics} (Diagrammatics) The morphisms $\psi$, $\psi'$, $\phi$, and $\phi'$ appearing in \eqref{eq: B_big short exact sequence SU(2,1)}, \eqref{eq: B_big in SU(2,1) via the other Bott-Samelson}, \eqref{eq: R idempotent in SU(2,1)} and \eqref{eq: R_t idempotent in SU(2,1)} have natural incarnations as string diagrams. This is how we came up with them. A full diagrammatic description of the morphism spaces of the categories in this paper will appear in future work by the third author. 
\end{remark}

We conclude from these computations that the second iteration of the recipe in \S\ref{sec: a recipe for finding indecomposable objects} yields one new ``layer 2'' indecomposable object: 
\[
B_{big}.
\]

\subsubsection{Layer 3}
\label{sec: layer 3 SU(2,1)}

We have seen in the preceding section that the big indecomposable object in our category occurs as the only new indecomposable object in layer 2. This means that our search for indecomposable objects has finished. We explain why below. 

The following argument holds for an arbitrary Coxeter system $(W,S)$ and reflection subgroup $W_K \subset W$. For $s \in S$, $R$ is a free right $R^s$-module with basis $\{1, \alpha_s\}$. Hence as an $(R^s, R^s)$-bimodule, 
\begin{equation}
    \label{eq: odd even decomposition}
    R \simeq R^s \oplus R^s(-2).
\end{equation}
Set $B_{big} = R^K \otimes_{R^W} R(k)$, where $k \in \Z_+$ is any shift\footnote{The appropriate shift for a given Lusztig--Vogan category depends on the rank of $R^K$ as a free $R^W$-module and the length of the longest element in $W$, if $W$ is finite. But any shift works for this argument.}. Using \eqref{eq: odd even decomposition}, we see that we have the following isomorphisms in $(R^K, R)$-gBim: 
\begin{align}
B_{big}B_s &\simeq R^K \otimes_{R^W} R \otimes_{R^s} R(k+1) \\
&\simeq R^K \otimes_{R^W} (R^s \oplus R^s(-2)) \otimes_{R^s} R (k+1) \\
\label{eq: B_s doubles B_big}
&\simeq B_{big}(-1) \oplus B_{big}(1). 
\end{align}
This computation shows that action on $B_{big}$ by any $B_s$ will not result in any new indecomposables in $\mc{N}_{LV}^0$.  

Applying \eqref{eq: B_s doubles B_big} to our example, we conclude that no new indecomposable objects arise from acting on $B_{big}$ with $B_s$ or $B_t$. Because $B_{big}$ was the only layer 2 indecomposable, this indicates that we have found all indecomposable objects in our category. All of our computations can be summarised using a version of the $W$-graph of the Lusztig--Vogan category of $SU(2,1)$, see Figure \ref{fig: W-graph of SU(2,1)}. 

\begin{figure}
\[
\begin{tikzcd}[row sep=2.5em, column sep=3em]
& & B_{big} \arrow[loop left, red, out=160, in=120, looseness=8] \arrow[loop right, blue, out=20, in=60, looseness=8] & & \\
& RB_t \arrow[dl, red, bend right] \arrow[ur, red]  \arrow[loop left, blue, out=160, in=120, looseness=6]& & 
  R_tB_s \arrow[dr, blue, bend right] \arrow[ul, blue] & \\
R \arrow[loop left, looseness=8, red] \arrow[ur, bend right, blue]& & R_t \arrow[ur, red] \arrow[ul, blue] & & R_{ts} \arrow[loop right, looseness=6, blue] \arrow[ul, bend right, red]
\end{tikzcd}
\]
\caption{A picture of $\mc{N}^0_{LV}(\SU(2,1))$. Nodes are indecomposable objects. A red (resp. blue) arrow from object $Q$ to object $Q'$ indicates $Q'$ is a direct summand of $QB_s$ (resp. $QB_t$).}  
\label{fig: W-graph of SU(2,1)}
\end{figure}
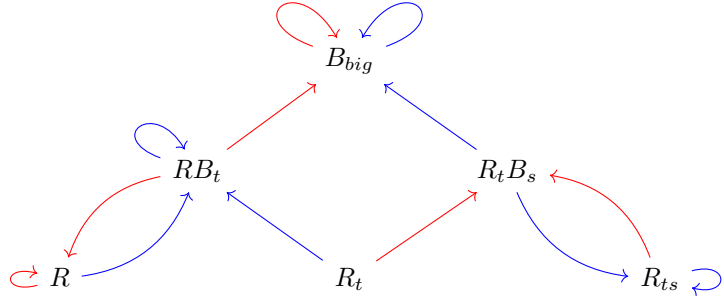
\section{A type $C$ Lusztig--Vogan category}
\label{sec: A type C Lusztig--Vogan category}
Our second example corresponds to a real form of the symplectic group $\Sp_4(\C)$. In this example, the subgroup $W_K \subset W$ is generated by one simple reflection and one reflection which is not simple. (In particular, it is not a parabolic subgroup.) To classify indecomposable objects in this category, we combine the direct computation method of \S\ref{sec: a recipe for finding indecomposable objects} with our second set of tools: the structure coming from the acting monoidal category. 

\subsection{The group $\Sp(1,1)$}
\label{sec: the group Sp(1,1)}
Let $G=\Sp_4(\C)$ be the group of linear transformations of $\C^4$ preserving the symplectic form 
\begin{equation}
    \label{eq: symplectic form}
(v,w):= v^T \bp 0 & 0 & 0 & 1 \\ 0 & 0 & 1 & 0 \\ 0 & -1 & 0 & 0 \\ -1 & 0 & 0 & 0 \ep w, 
\end{equation}
for $v, w \in \C^4$.
Define a Cartan involution on $G$ by 
\begin{equation}
    \label{eq: Cartan involution type C}
\theta: g \mapsto \bp -1 & 0 & 0 & 0 \\ 0 & 1 & 0 & 0 \\ 0 & 0 & 1 & 0 \\ 0 & 0 & 0 & -1 \ep g \bp -1 & 0 & 0 & 0 \\ 0 & 1 & 0 & 0 \\ 0 & 0 & 1 & 0 \\ 0 & 0 & 0 & -1 \ep.
\end{equation}
The real group $\Sp(1,1)$ corresponding to $\theta$ is the set of linear transformations of $\C^4$ preserving both the symplectic form \eqref{eq: symplectic form} and the Hermitian form 
\begin{equation}
    \label{eq: Hermitian form}
    \langle v, w \rangle := v^T \bp -1 & 0 & 0 & 0 \\ 0 & 1 & 0 & 0 \\ 0 & 0 & 1 & 0 \\ 0 & 0 & 0 & -1 \ep \overline{w},
\end{equation}
for $v, w \in \C^4$. 

The elements of $G$ fixed by \eqref{eq: Cartan involution type C} form the group 
\[
K :=G^\theta= \left\{ \bp * & 0 & 0 & * \\ 0 & * & * & 0 \\ 0 & * & * & 0 \\ * & 0 & 0 & * \ep \right\} \subset \Sp_4(\C).
\]
Here the $*$'s denote potentially non-zero matrix entries. An easy exercise using the symplectic form \eqref{eq: symplectic form} shows that $K \simeq \SL_2(\C) \times \SL_2(\C)$, with one copy of $\SL_2(\C)$ in the middle $2 \times 2$ block and the other occurring at the corners. The groups $G$ and $K$ share a diagonal rank 2 maximal torus, so the group $\Sp(1,1)$ is also of equal rank. 

The Weyl group $W$ of $G$ is the dihedral group $D_4$:
\[
W = D_4= \langle s, t, \mid s^2=t^2=1, (st)^4=1 \rangle.
\]
The root system associated to $G$ is of type $C_2 = B_2$ (Figure \ref{fig: C2 root system}).
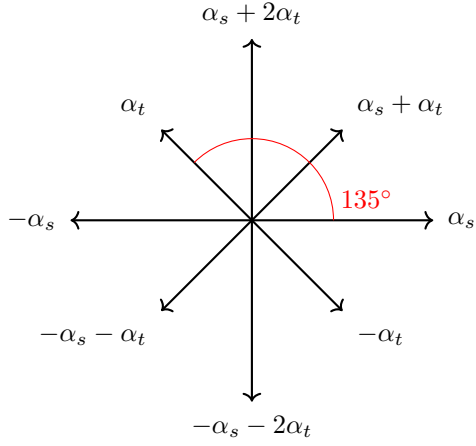
\begin{figure}
\begin{tikzpicture}[scale=1.2, thick]
    \path[->] 
    (0,0) edge ( 2, 0)
    (0,0) edge (-2, 0)
    (0,0) edge ( 0, 2)
    (0,0) edge ( 0,-2)
    ;
    \node[right] at ( 2.05, 0) {$\alpha_s$};
    \node[left]  at (-2.05, 0) {$-\alpha_s$};
    \node[above] at ( 0, 2.05) {$\alpha_s + 2\alpha_t$};
    \node[below] at ( 0,-2.05) {$-\alpha_s - 2\alpha_t$};
    \path[->] 
    (0,0) edge (-1, 1)
    (0,0) edge ( 1,-1)
    (0,0) edge (-1,-1)
    (0,0) edge ( 1, 1)
    ;
    \node[above right] at ( 1.05, 1.05) {$\alpha_s + \alpha_t$};
    \node[above left]  at (-1.05, 1.05) {$\alpha_t$};
    \node[below right] at ( 1.05,-1.05) {$-\alpha_t$};
    \node[below left]  at (-1.05,-1.05) {$-\alpha_s - \alpha_t$};
    \draw[thin, red] (0.9, 0) arc[radius=0.9, start angle=0, end angle=135];
    \node[right, red] at (15:0.9) {$135^\circ$};
\end{tikzpicture}
\caption{The root system of type $C_2=B_2$}
\label{fig: C2 root system}
\end{figure}

Under the embedding $W_K \hookrightarrow W$, the Weyl group $W_K\simeq \Z/2\Z \times \Z/2\Z$ must be generated by reflections with respect to orthogonal root vectors. Without loss of generality\footnote{Note that the two pairs of orthogonal roots in Figure \ref{fig: C2 root system} are not $W$-conjugate, so they do not correspond to conjugate Cartan involutions, see Remark \ref{rem: conjugate real forms}. However, there is an automorphism of the Coxeter diagram of type $C_2=B_2$ which swaps $s$ and $t$ and induces an equivalence of the corresponding Lusztig--Vogan categories, so our choice of short or long generators does not influence the categorical structure.}, we may assume these to be the long roots $\alpha_s$ and $\alpha_s + 2\alpha_t$. Reflection across the hyperplane orthogonal to $\alpha_s + 2 \alpha_t$ is given by the element $tst \in W$, hence 
\[
W_K = \langle s, tst \rangle \subset W. 
\]
The quotient $W_K \backslash W$ consists of two cosets, with minimal coset representatives 
\begin{equation}
    \label{eq: min length Sp(1,1)}
    {}^K W = \{1, t\}. 
\end{equation}

As in the previous example, we set $V$ to be the $\R$-span of $\alpha_s$ and $\alpha_t$, which inherits the natural action of $W$. In this example, it is convenient to work with a basis for $V$ which does not consist entirely of simple roots. Set 
\begin{equation}
    \label{eq: x and y}
    x:=\alpha_s \text{ and } y:= \alpha_s + 2 \alpha_t,
\end{equation}
and identify $R = \op{Sym}(V)$ with the polynomial ring $\R[x, y]$. Again, we view this as a graded ring with $\deg V = 2$. We have 
\begin{equation}
    \label{eq: s and tst action in type C}
    s(x) = -x, \hspace{2mm} s(y) = y, \hspace{2mm} tst(x) = x, \hspace{1mm} \text{ and } \hspace{1mm} tst(y) = -y.
\end{equation}

\begin{remark} 
\label{rem: realisation} In our first example in \S\ref{sec: A type A Lusztig--Vogan category}, the root space representation $V$ coincided with the geometric representation, see \cite[\S5.3]{HumphreysCoxeter} or \cite[\S5.7]{SBim}. In this example, $V$ differs from the geometric representation because the root system of $\Sp_4(\C)$ has roots of different lengths. For our purposes, it is natural to use the root space representation, but one should note that this choice changes some formulas in \cite{SBim}, such as reflection formulas \cite[(4.5)]{SBim} and coefficients of Jones--Wenzl projectors \cite[\S 9.3]{SBim}. With either choice of realisation (geometric or root space), the resulting Lusztig--Vogan categories are equivalent.
\end{remark}

The Chevalley--Shephard--Todd theorem guarantees that $R^K$ is generated by two degree $2$ generators. We see from \eqref{eq: s and tst action in type C} that $x^2$ and $y^2$ suffice:
\begin{equation}
    \label{eq: R^K for type C}
    R^K = \R[x^2, y^2]. 
\end{equation}
Some other invariant subrings of $R$ will play a role in our computations. We list them here for future reference: 
\begin{equation}
    \label{eq: type C invariant subrings}
    R^s = \R[x^2, y], \hspace{2mm} R^t = \R[x+y, xy], \hspace{2mm} R^W = \R[x^2+y^2, x^2y^2].
\end{equation}

\subsection{The category $\mc{N}_{LV}^0(\Sp(1,1))$}
\label{sec: The category N_LV for Sp(1,1)} We will now classify indecomposable objects in the Lusztig--Vogan category corresponding to the real group $\Sp(1,1)$. 

\subsubsection{Generators}
\label{sec: generators Sp(1,1)} The category $\mc{N}_{LV}^0$ has two generators, corresponding to the two minimal coset representatives in $^K W$:
\[
\mc{N}_{LV}^0 = \langle \{ R, R_t\} \otimes_R \SBim \rangle_{\oplus, \ominus, (1)} \subset (R^K, R)\text{-gBim}
\]
As in the previous   example, these generators are indecomposable: 
\begin{equation}
    \label{eq: generators Sp(1,1)}
    R, R_t.
\end{equation}

\subsubsection{Layer 1}
\label{sec: layer 1 Sp(1,1)}
Again, we proceed following the recipe in \S \ref{sec: a recipe for finding indecomposable objects}. The action on the generators \eqref{eq: generators Sp(1,1)} by $B_s$ and $B_t$ results in four objects:
\begin{equation}
    \label{eq: first list layer 1 Sp(1,1)}
    RB_s, \hspace{2mm} RB_t, \hspace{2mm} R_t B_s, \text{ and } R_t B_t. 
\end{equation}
The isomorphism lemma (Lemma \ref{lem: isomorphism lemma}) immediately implies one isomorphism:
\begin{equation}
    \label{eq: Bt is RtBt for Sp(1,1)}
    RB_t \simeq R_t B_t. 
\end{equation}
From the actions \eqref{eq: s and tst action in type C} on the invariant subrings \eqref{eq: type C invariant subrings}, one can see that the splitting lemma (Lemma \ref{lem: splitting lemma}) applies to $RB_s$ and $R_tB_s$: 
\begin{equation}
    \label{eq: splittings in Sp(1,1) layer 1}
    RB_s \simeq R^{\oplus v^{-1}+v} \text{ and } R_t B_s \simeq R_t^{\oplus v^{-1} + v}.
\end{equation}

This leaves us with one isomorphism class of ``Bott--Samelson-type'' objects in layer 1: 
\[
RB_t.
\]

We wish to show this object is indecomposable, but this is not as straightforward as it was in the previous example. Here we encounter our first indecomposable module where our easy indecomposability criterion (Lemma \ref{lem: indecomposability via generation by homogeneous elt}) does not apply: one can see from looking at the invariant subrings in \eqref{eq: R^K for type C} and \eqref{eq: type C invariant subrings} that $R^K$ and $R^t$ do not generate $R$, so $RB_t$ is not generated by $1 \otimes 1$.  

At this point, there are two ways to proceed. Our second indecomposability criterion (Lemma \ref{lem: indecomposability lemma}) does apply in this setting. (In fact, it was formulated exactly to work in this example.) However, one can also prove indecomposability of $RB_t$  using the method of \S\ref{sec: leaning into the monoidal category} by exploiting the known structure of the category of Soergel bimodules. We will use this case to illustrate the second method. 

Assume that $RB_t$ does decompose. The exact sequence \eqref{eq: initial ses} shows that the standard factors of $RB_t$ are $R$ and $R_t$ (each with multiplicity 1). The standard factors of a direct sum are the union of the standard factors of each summand, so if $RB_t$ decomposes, it must decompose as 
\begin{equation}
\label{eq: decomposition of RBt}
RB_t = R(k) \oplus R_t(\ell),
\end{equation}
for shifts $k, \ell \in \Z$. Now, we know from \eqref{eq: splittings in Sp(1,1) layer 1} that right action by $B_s$ doubles (up to shift) both $R$ and $R_t$. Combining this with the $B_t$-action on $R$ and $R_t$ given by the decomposition \eqref{eq: decomposition of RBt} and the isomorphism \eqref{eq: Bt is RtBt for Sp(1,1)}, we see that any Bott--Samelson bimodule with left action restricted to $R^K$ must decompose as a sum of shifts of $R$ and $R_t$:
\[
RBS(\underline{w}) = R^{\oplus h} \oplus R_t^{\oplus k }
\]
for $h, k \in \Z[v^{\pm 1}]$. In particular, this implies that the only indecomposable objects (up to shift) in our category are $R$ and $R_t$. However, we know that this is not the case, because we also have the big indecomposable object. 

The big indecomposable object is the $(R^K, R)$-bimodule
\begin{equation}
    \label{eq: the big indecomposable in Sp(1,1)}
    B_{big}:= R^K \otimes_{R^W} R (2). 
\end{equation}
As in Section \ref{sec: A type A Lusztig--Vogan category}, this object is indecomposable (as it is generated by $1 \otimes 1$), and it is in our category because it is a summand of the big Soergel bimodule. In this example, we will investigate this relationship more closely.  

Recall that for a finite Coxeter system $W$, the big Soergel bimodule is the $(R, R)$-bimodule
\[
B_{w_0} := R \otimes_{R^W}R(\ell (w_0)), 
\]
where $w_0 \in W$ is the longest element. This indecomposable module appears as the unique\footnote{What we mean by ``unique'' is somewhat subtle here. Direct sum decompositions of $BS(\underline{w_0})$ into indecomposable summands are not canonical; however, by the Krull--Schmidt property of the category, any such decomposition will contain the same summands with the same multiplicities, up to isomorphism. The indecomposable summand $B_{w_0}$ is characterised by the property that in any direct sum decomposition of $BS(\underline{w_0})$ into indecomposable summands, the summand containing $1 \otimes \cdots \otimes 1$ is isomorphic to $B_{w_0}$.} indecomposable summand containing $1 \otimes \cdots \otimes 1$ in a Bott--Samelson bimodule associated to a reduced expression of $w_0$, so it is an indecomposable Soergel bimodule. 

By the Chevalley--Shephard--Todd theorem, $R$ is a free finite-rank $R^K$-module for any reflection subgroup $W_K \subset W$. In our current example, a basis is $\{1, x, y, xy\}$, so $R$ is free over $R^K$ of graded rank $1 + 2v^{-2} + v^{-4}$. Hence, as graded $(R^K, R)$-bimodules, we have a decomposition 
\begin{align}
    RB_{w_0} &\simeq R \otimes_{R^W} R (4)\\
    &= (R^K)^{\oplus 1 + 2v^{-2} + v^{-4}} \otimes_{R^W} R (4) \\
    \label{eq: decomposition of big Soergel bimodule into bid indecomposable}
    &= B_{big}^{\oplus v^{-2} + 2 + v^2}.
\end{align}
This implies that $B_{big}$ is in our category. 

Moreover, one can show through iterated applications of \eqref{eq: initial ses} that the $R$-bimodule $B_{w_0}$ has each of the eight standard $R$-bimodules appearing in its standard filtration exactly once (up to shift). Hence \eqref{eq: decomposition of big Soergel bimodule into bid indecomposable} shows that $B_{big}$ has two standard factors: $R$ and $R_t$, each occurring with multiplicity 1. This implies (by Lemma \ref{lem: standard factors}) that $B_{big}$ is not isomorphic to $R$ or $R_t$, contradicting our claim that the only indecomposable objects are $R$ and $R_t$. We conclude that $RB_t$ must be indecomposable. 

For any finite Coxeter group $W$ and reflection subgroup $W_K \subset W$, we have a decomposition like \eqref{eq: decomposition of big Soergel bimodule into bid indecomposable}, so the big indecomposable object is always in the Lusztig--Vogan category. In particular, this implies that we have at least $|W_K \backslash W| + 1$ indecomposable objects in any Lusztig--Vogan category--the generators $R_x$ for $x \in {^KW}$ and $B_{big}$. 

\subsubsection{Layer 2}
\label{sec: layer 2 Sp(1,1)} There are two layer 2 Bott--Samelson-type objects obtained from tensoring $RB_t$ with $B_s$ and $B_t$:
\[
RB_t B_t \text{ and } RB_t B_s. 
\]
The first one decomposes as
\[
RB_t B_t \simeq (RB_t)^{\oplus v^{-1} + v},
\]
so there is one remaining layer 2 object to investigate. Again, can use our knowledge of the monoidal structure of the category $\SBim(C_2)$ to decompose this in two different ways.

Using the decomposition of $RB_s$ in \eqref{eq: splittings in Sp(1,1) layer 1}, we can decompose the restricted Bott--Samelson bimodule $RB_sB_tB_s$ in $(R^K,R)$-gBim as follows:
\begin{equation}
    \label{eq: decomposition of RBsBtBs in Sp(1.1)}
    RB_sB_tB_s \simeq  (R^{\oplus v^{-1} + v}) B_t B_s \simeq (RB_t B_s)^{\oplus v^{-1} + v}.
\end{equation}

On the other hand, we have the following decomposition of the unrestricted Bott--Samelson bimodule $B_sB_t B_s$ in $\SBim$:
\begin{equation}
    \label{eq: decomposition of BsBtBs}
    B_sB _tB_s \simeq B_{sts} \oplus B_s.
\end{equation}
Restricting the left $R$-action on \eqref{eq: decomposition of BsBtBs} to $R^K$ and decomposing $RB_s$ by \eqref{eq: splittings in Sp(1,1) layer 1}, we have 
\begin{equation}
    \label{eq: second decomposition of RB_sBtBs for Sp(1,1)}
    RB_sB_tB_s \simeq RB_{sts} \oplus R^{\oplus v^{-1} + v}.
\end{equation}
The indecomposable summands occurring in \eqref{eq: decomposition of RBsBtBs in Sp(1.1)} and \eqref{eq: second decomposition of RB_sBtBs for Sp(1,1)} must agree by the Krull--Schmidt property, so we conclude that (a shift of) $R$ is a summand of $RB_tB_s$. 

An analogous argument using the decomposition of $R_tB_s$ in \eqref{eq: splittings in Sp(1,1) layer 1} and the isomorphism \eqref{eq: Bt is RtBt for Sp(1,1)} lets us conclude that (a shift of) $R_t$ is also a summand of $RB_tB_s$. Hence we have a decomposition 
\begin{equation}
    \label{eq: decomposition of RBtBs into indecomposable summands}
    RB_tB_s \simeq R(k) \oplus R_t(\ell) \oplus B,
\end{equation}
for some $(R^K, R)$-bimodule $B$ and shifts $k, \ell \in \Z$.

It turns out that the indecomposable summand $B$ is isomorphic to $B_{big}$ and the shifts $k=\ell=0$. Indeed, there is a short exact sequence in $(R^K, R)$-gBim
\begin{equation}
    \label{eq: B_big ses Sp(1,1)}
    0 \longrightarrow B_{big} \xlongrightarrow{\iota} RB_tB_s \xlongrightarrow{\psi} R \oplus R_t \longrightarrow 0
\end{equation}
where $\iota(1 \otimes 1) = 1 \otimes 1 \otimes 1$, and 
\[
\psi(f \otimes g \otimes h) = \left( \partial_s(fg) h, -t\partial_s(t(f)g)t(h) \right).
\]
As in \S\ref{sec: A type A Lusztig--Vogan category}, we are identifying $RB_t B_s$ with $R \otimes_{R^t} R \otimes_{R^s} R (2)$. The map 
\begin{align}
    \label{eq: splitting of R and R_t from RB_tB_s in Sp(1,1) 1}
    R \oplus R_t &\longrightarrow RB_tB_s \\
    \label{eq: splitting of R and R_t from RB_tB_s in Sp(1,1) 2}
    (1,0) &\mapsto -c_t \otimes 1 \\
    \label{eq: splitting of R and R_t from RB_tB_s in Sp(1,1) 3}
    (0,1) &\mapsto -d_t \otimes 1
\end{align}
provides a splitting of \eqref{eq: B_big ses Sp(1,1)}. 

We conclude that there is one new indecomposable obtained at layer 2:
\[
B_{big}.
\]

\subsubsection{Layer 3}
\label{sec: layer 3 Sp(1,1)}
By the argument in \S\ref{sec: layer 1 Sp(1,1)}, we have 
\[
B_{big} B_s \simeq B_{big}^{\oplus v^{-1} + v} \simeq B_{big}B_t,
\]
so our algorithm for finding indecomposable objects terminates at layer 3. We have recovered a version of the $W$-graph for the Lusztig--Vogan module of $\Sp(1,1)$, see Figure \ref{fig: W-graph of Sp(1,1)}.

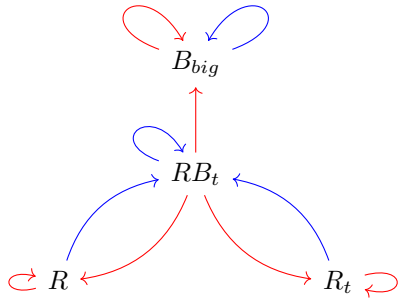
\begin{figure}
\[
\begin{tikzcd}[row sep=2.5em, column sep=3em]
& B_{big} \arrow[loop left, red, out=160, in=120, looseness=8] \arrow[loop right, blue, out=20, in=60, looseness=8] & \\
& RB_t \arrow[u, red]  \arrow[loop left, blue, out=160, in=120, looseness=6] \arrow[dl, bend left, red] \arrow[dr, bend right, red] \\
R \arrow[loop left, looseness=8, red] \arrow[ur, bend left, blue]& & R_t \arrow[loop right, looseness=8, red] \arrow[ul, bend right, blue] 
\end{tikzcd}
\]
\caption{A picture of $\mc{N}^0_{LV}(\Sp(1,1))$. Nodes are indecomposable objects. A red (resp. blue) arrow from object $Q$ to object $Q'$ indicates $Q'$ is a direct summand of $QB_s$ (resp. $QB_t$). }
\label{fig: W-graph of Sp(1,1)}
\end{figure}

\section{A type $G$ Lusztig--Vogan category}
\label{sec: A type G Lusztig--Vogan category}

Our third example is the split real form of complex $G_2$. In this example, the group $W_K$ is generated by two reflections which are not simple. This example is considerably more complicated than the first two examples, and the computations become intractable by hand. However, they can be done with the help of a computer. In this section, we describe the split real form of $G_2$ and illustrate how far we can go in describing its Lusztig--Vogan category using our existing tools. The final enumeration of indecomposables is completed in \S\ref{sec: Implementation for split G2}, after we introduce an algorithm in \S\ref{sec: an algorithm to compute indecomposable objects} which implements the recipe in \S \ref{sec: a recipe for finding indecomposable objects}. 

\subsection{The group split $G_2$}
\label{sec: the group split G_2}

Let $G$ be the complex simple Lie group of type $G_2$. There are several ways to model this group. We refer the curious reader to \cite{Draper} for a nice survey. Our preferred way of engaging with this group is through its incarnation via the triality automorphism of the $D_4$ Dynkin diagram. This allows us to explicitly describe it as a group of $8 
\times 8$ matrices\footnote{An explicit description of $G_2$ is not necessary for the construction of the Lusztig--Vogan category, but we feel that it does this beautiful mathematical object an injustice not to dedicate a bit of space to exploring it.}. We briefly describe this perspective below before discussing its real form. We work on the level of Lie algebras. 

The special orthogonal Lie algebra $\mathfrak{so}(8)$ can be realised as the set of traceless linear transformations of $\C^8$ preserving the quadratic form\footnote{This choice of quadratic form has the feature that it embeds  $\mf{so}(8)$ in $\mf{sl}_8(\C)$ in such a way that its intersection with the upper triangle Borel subalgebra of $\mf{sl}_8(\C)$ is a Borel subalgebra of $\mf{so}(8)$.} given by the anti-diagonal matrix 
\[
J = 
\bp 0 & & 1 \\
 & \iddots & \\
1 &  & 0 \ep.
\]
This allows us to view elements in $\mf{so}(8)$ as traceless $8 \times 8$ matrices which are skew-symmetric across the antidiagonal. In particular, with respect to the diagonal Cartan subalgebra, the root spaces of $\mf{so}(8)$ are spanned by the $24$ matrices 
\begin{equation}
    \label{eq: chevalley basis so(8)}
    \{e_{ij} - e_{8-j+1,8-i+1} \mid i+j<9, i\neq j \},
\end{equation} 
where $e_{ij}$ is the matrix with a $1$ in the $ij$ position and zeros elsewhere. The $12$ positive root spaces are those with $i<j$, and the four simple root spaces are spanned by
\begin{equation}
    \label{eq: simple root vectors so(8)}
    X_{\alpha_1}:=e_{12}-e_{78}, \hspace{2mm} X_{\alpha_2}:=e_{23}-e_{67}, \hspace{2mm}  X_{\alpha_3}:=e_{34}-e_{56}, \text{ and } X_{\alpha_4}:=e_{35}-e_{46}. 
\end{equation}
The set \eqref{eq: chevalley basis so(8)} can be extended to a Chevalley basis by including $H_{\alpha_i}:=[X_{\alpha_i},X_{\alpha_i}^T]$ for $i=1, \dots, 4$. 

The Dynkin diagram of $\mf{so}(8)$ is of type $D_4$. It has an automorphism (``triality'') of order three which fixes $\alpha_2$ and cyclically permutes $\alpha_1, \alpha_3,$ and $\alpha_4$: 
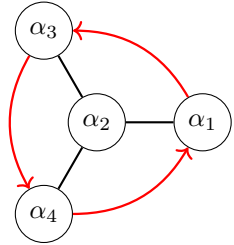
\[
\begin{tikzpicture}[scale=0.7]
    \node (A) at (0, 0) [circle, draw] {\(\alpha_2\)};
    \node (B) at (2, 0) [circle, draw] {\(\alpha_1\)};
    \node (C) at ({2*cos(120)},{2*sin(120)}) [circle, draw] {\(\alpha_3\)};
    \node (D) at ({2*cos(240)},{2*sin(240)}) [circle, draw] {\(\alpha_4\)};
    
    \draw[thick] (A) -- (B);
    \draw[thick] (A) -- (C);
    \draw[thick] (A) -- (D);

    \draw[->, red, bend right, line width=0.3mm] (B) to (C);
    \draw[->, red, bend right, line width=0.3mm] (C) to (D);
    \draw[->, red, bend right, line width=0.3mm] (D) to (B);
\end{tikzpicture}
\]
This induces a Lie algebra automorphism $\mathrm{Tri}:\mf{so}(8) \rightarrow \mf{so}(8)$ by assigning
\begin{equation}
    \label{eq: triality on so(8)}
    \mathrm{Tri}(X_{\alpha_i}) = X_{\mathrm{Tri}(\alpha_i)}
\end{equation}
for simple roots $\alpha_1, \ldots , \alpha_4$ and extending via the Lie bracket. Because the root vectors in \eqref{eq: chevalley basis so(8)} are Chevalley basis vectors, $\mathrm{Tri}$ sends a vector on this list to a scalar multiple of another vector on the list. The scalars are given by the structure constants of the Chevalley basis. In this way, one can give a description of the action of the automorphism $\mathrm{Tri}$ on the root spaces of $\mf{so}(8)$ using a collection of labelled arrows between certain entries in the first four rows of a matrix in $\mf{so}(8)$. This is done in Figure \ref{fig: triality on so(8)}. The matrix entries chosen correspond to the positive part of the positive root vectors in \eqref{eq: chevalley basis so(8)}. It is sufficient to illustrate the action on the first four rows because matrices in $\mf{so}(8)$ are skew-symmetric across the antidiagonal. 

Combining Figure \ref{fig: triality on so(8)} with the action on the remaining Chevalley basis vectors \{$H_{\alpha_i}; i=1, \ldots, 4\}$ results in a complete description of the automorphism $\mathrm{Tri}$ on $\mf{so}(8)$:  
\begin{equation}
\label{eq: triality action on Cartan}
   H_{\alpha_1} \xmapsto{\mathrm{Tri}} H_{\alpha_3}, \hspace{2mm} H_{\alpha_2} \xmapsto{\mathrm{Tri}} H_{\alpha_2}, \hspace{2mm} H_{\alpha_3} \xmapsto{\mathrm{Tri}} H_{\alpha_4}, \hspace{2mm}H_{\alpha_4} \xmapsto{\mathrm{Tri}} H_{\alpha_1}.
\end{equation}
\begin{figure}
\[
\begin{tikzpicture}[scale=1.6]
    \node at (3.5,1.5) {$\left( \phantom{\rule{340pt}{87pt}} \right)$}; 
    \node (03) at (0,3) {$\cdot$};
    \node (05) at (0,2) {$\cdot$};
    \node (06) at (0,1) {$\cdot$};
    \node (07) at (0,0) {$\cdot$};
    \node (08) at (1,0) {$\cdot$};
    \node (09) at (1,1) {$\cdot$};    
    \node (09) at (2,0) {$\cdot$};   
    \node (09) at (5,0) {$\cdot$};        
    \node (09) at (6,0) {$\cdot$};        
    \node (09) at (7,0) {$\cdot$};        
    \node (09) at (7,1) {$\cdot$};        
    \node (09) at (6,1) {$\cdot$};        
    \node (09) at (7,2) {$\cdot$};        
    
    \node (13) at (1,3) {$\color{blue} \blacksquare$};
    \node (23) at (2,3) {$\bigstar$};
    \node (33) at (3,3) {$\bigstar$};
    \node (43) at (4,3) {$\bigstar$};
    \node (53) at (5,3) {$\bigstar$};
    \node (63) at (6,3) {$\bigstar$};
    \node (73) at (7,3) {$0$};
    \node (02) at (0,2) {$$};
    \node (12) at (1,2) {$\cdot$};
    \node (22) at (2,2) {$\color{blue} \blacksquare$};
    \node (32) at (3,2) {$\bigstar$};
    \node (42) at (4,2) {$\bigstar$};
    \node (52) at (5,2) {$\bigstar$};
    \node (62) at (6,2) {$0$};
    \node (01) at (0,1) {$$};
    \node (11) at (1,1) {$$};
    \node (21) at (2,1) {$\cdot$};
    \node (31) at (3,1) {$\color{blue} \blacksquare$};
    \node (41) at (4,1) {$\color{blue} \blacksquare$};
    \node (51) at (5,1) {$0$};
    \node (00) at (0,0) {$$};
    \node (10) at (1,0) {$$};
    \node (20) at (2,0) {$$};
    \node (30) at (3,0) {$\cdot$};
    \node (40) at (4,0) {$0$};
    \path[->, red, thick]
        (53) edge[out=120, in=60, looseness=8] (53)
        (63) edge[out=120, in=60, looseness=8] (63)
        (22) edge[out=140, in=80, looseness=8] (22)
        %
        (43) edge[out=180-15, in=15, looseness=1] (33)
        (33) edge (52)
        (52) edge[out=115, in=-25, looseness=1] (43)
        %
        (23) edge node[right]{$-1$} (32)
        (32) edge (42)
        (42) edge[out=135, in=-5, looseness=1] node[below, yshift=-2pt]{$-1$} (23)
        %
        (13) edge[out=-60, in=140, looseness=1] (31)
        (31) edge (41)
        (41) edge[out=140, in=-20, looseness=1] (13)
    ;
\end{tikzpicture}
\]
    \caption{The triality automorphism on root vectors of $\mf{so}(8)$. {\color{blue} Blue $\blacksquare$} entries correspond to simple root spaces, and  black $\bigstar$ entries to the remaining positive root vectors in \eqref{eq: chevalley basis so(8)}. When the action scales by a Chevalley structure constant, the scalar is listed below the arrow. Unlabelled arrows correspond to the constant $1$.}
 \label{fig: triality on so(8)}

\end{figure}

The complex Lie algebra of type $G_2$ arises as the fixed points under the triality automorphism\footnote{Here is a fun exercise: one can verify that the resulting Lie subalgebra is indeed of type $G_2$ by orthogonally projecting the root system of $\mf{so}(8)$ onto the subspace of $\R^4$ spanned by $\alpha_s:=\alpha_2$ and $\alpha_t:= \frac{1}{3}(\alpha_1 + \alpha_3 + \alpha_4)$, which are both fixed by $\mathrm{Tri}$, and showing that the resulting set of vectors forms the root system of type $G_2$ in Figure \ref{fig: G2 root system}.} :
\begin{equation}
    \label{eq: definition of g_2}
    \mf{g}_2:= \mf{so}(8)^{\mathrm{Tri}}.
\end{equation}
Using Figure \ref{fig: triality on so(8)}, one can immediately determine a Chevalley basis for $\mf{g}_2$.
\begin{lemma} 
\label{lem: chevalley basis of g2}
 The following matrices in $\mf{so}(8)$ form a Chevalley basis of $\mf{g}_2$ as a subalgebra of $\mf{so}(8)$:
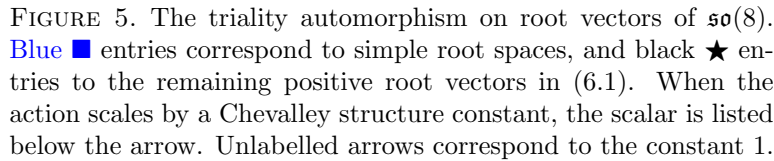
\begin{enumerate}
    \item the three root vectors for $\mf{so}(8)$ which are fixed by $\mathrm{Tri}$: 
    \[
    \hspace{2mm}  e_{23}-e_{67}, \hspace{2mm} e_{16} - e_{38}, \hspace{2mm} e_{17} - e_{28},
    \]
    \item appropriately scaled sums of the three loops in Figure \ref{fig: triality on so(8)}:
    \[
    (e_{12} - e_{78}) + (e_{34} - e_{56}) + (e_{35} - e_{46}),
    \]
    \[
    (e_{13}  - e_{68}) - (e_{24} - e_{57})  - (e_{25} - e_{47}),  
    \]
    \[
    (e_{14} - e_{58}) + (e_{26} - e_{37}) + ( e_{15} - e_{48}),
    \]
    \item the transposes of the vectors in (1) and (2); and 
    \item the fixed Cartan basis vector, and the sum of the Cartan loop in \eqref{eq: triality action on Cartan}:
    \[
    (e_{22} - e_{77}) - (e_{33} - e_{66}), \hspace{2mm} (e_{11}-e_{88}) - (e_{22} - e_{77}) +2(e_{33}-e_{66}).
    \]
\end{enumerate}
\end{lemma}
This Chevalley basis consists of 14 vectors, giving a concrete description of the Lie algebra of type $\mf{g}_2$ as the $\C$-span of 14 matrices in $\mf{so}(8)$. 

In Lemma \ref{lem: chevalley basis of g2}, the simple root spaces are those spanned by 
\[
X_{\alpha_s}:= e_{23}-e_{67} \text{ (long) }\text{ and } X_{\alpha_t}:=     (e_{12} - e_{78}) + (e_{34} - e_{45}) + (e_{35} - e_{46}) \text{ (short)}.
\]
Taking brackets of these matrices and their transposes generates the remaining basis elements in Lemma \ref{lem: chevalley basis of g2} and provides an identification of the basis in Lemma \ref{lem: chevalley basis of g2} with $\{X_\alpha \mid \alpha \in R\}$, where $R$ is the root system of type $G_2$ (Figure \ref{fig: G2 root system}). Under this identification, the six strictly upper triangular matrices in Lemma \ref{lem: chevalley basis of g2} correspond to the six positive roots in $R^+$. 

The Cartan involution of $\mf{g}_2$ corresponding to the split real form\footnote{On the level of groups, this is the involution $g \mapsto {^Tg^{-1}}$ for $g \in G_2 \subset SO(8)$.} is 
\[
\theta: X \mapsto -X^T.
\]
The following six matrices form a basis for $\mf{k} = \mf{g}_2^\theta$:
\begin{equation}
    \label{eq: basis for K}
    \{(X_\alpha - X_{-\alpha}) \mid \alpha \in R^+\}.
\end{equation}
Direct computation\footnote{This is difficult by hand, but easily done using computational algebra software. We verified it using Magma.} shows that these matrices span a Lie subalgebra isomorphic to $\mf{sl}_2(\C) \times \mf{sl}_2(\C)$, and the corresponding connected subgroup of $G$ has Weyl group $W_K\simeq \Z/2\Z \times \Z/2\Z$.   

The Weyl group of $G$ is the dihedral group $D_6$:
\[
W =D_6= \langle s, t \mid s^2=t^2=1, (st)^6=1 \rangle.
\]
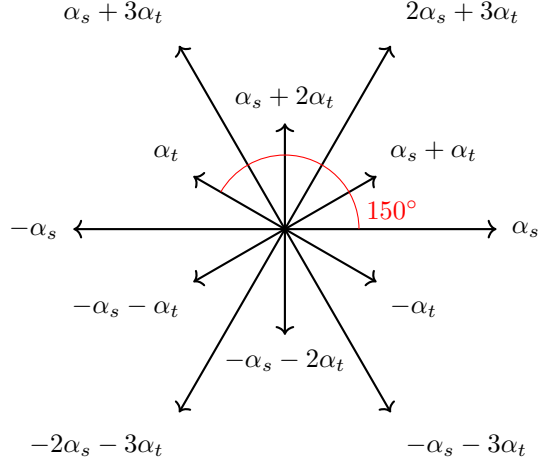
\begin{figure}
\begin{tikzpicture}[scale=1.4, thick]
    \path[->] 
    (0,0) edge ( 2, 0)
    (0,0) edge ( 1, 1.732)
    (0,0) edge (-1, 1.732)
    (0,0) edge (-2, 0)
    (0,0) edge (-1,-1.732)
    (0,0) edge ( 1,-1.732)
    ;
    \node[right] at (2.05,0) {$\alpha_s$};
    \node[above right] at (1.05,1.85) {$2\alpha_s+3\alpha_t$};
    \node[above left] at (-1.05,1.85) {$\alpha_s+3\alpha_t$};
    \node[left] at (-2.05,0) {$-\alpha_s$};
    \node[below left] at (-1.05,-1.85) {$-2\alpha_s-3\alpha_t$};
    \node[below right] at (1.05,-1.85) {$-\alpha_s-3\alpha_t$};
    \path[->] 
    (0,0) edge (-0.866, 0.5)
    (0,0) edge ( 0, 1)
    (0,0) edge ( 0.866, 0.5)
    (0,0) edge ( 0.866,-0.5)
    (0,0) edge ( 0,-1)
    (0,0) edge (-0.866,-0.5)
    ;
    \node[above left] at (-0.9,0.55) {$\alpha_t$};
    \node[above] at (0,1.05) {$\alpha_s+2\alpha_t$};
    \node[above right] at (0.9,0.55) {$\alpha_s+\alpha_t$};
    \node[below right] at (0.9,-0.55) {$-\alpha_t$};
    \node[below] at (0,-1.05) {$-\alpha_s-2\alpha_t$};
    \node[below left] at (-0.9,-0.55) {$-\alpha_s-\alpha_t$};
    \draw[thin, red] (0.7, 0) arc[radius=0.7, start angle=0, end angle=150];
    \node[red,right] at (15:0.7) {$150^\circ$};
\end{tikzpicture}
\caption{The root system of type $G_2$.}
\label{fig: G2 root system}
\end{figure}
Hence under the embedding $W_K \hookrightarrow W$, $W_K$ must be generated by a pair of orthogonal reflections. All such pairs are $W$-conjugate, so they correspond to conjugate involutions on $G$, and thus the same real form (Remark \ref{rem: conjugate real forms}). We choose to work with the subgroup 
\begin{equation}
    \label{eq: W_K for G2}
    W_K = \langle sts, tst \rangle \subset W 
\end{equation}
to illustrate the structure that arises when the reflection subgroup has no simple generators. The reflections $sts$ and $tst$ correspond to reflections across roots $\alpha_s + \alpha_t$ and $\alpha_s + 3 \alpha_t$, respectively.

The quotient $W_K \backslash W$ consists of three cosets, with minimal coset representatives 
\begin{equation}
    \label{eq: min length G2}
    {}^K W = \{1, s, t\}. 
\end{equation}

As in Example \ref{sec: the group Sp(1,1)}, we choose to work with a basis for $V=\mathrm{span}_\R\{\alpha_s, \alpha_t\}$ of orthogonal vectors to make computations easier. Set 
\begin{equation}
    \label{eq: x and y for G2}
    x:= \alpha_s + \alpha_t \text{ and } y:=\alpha_s + 3 \alpha_t,
\end{equation}
and identify $R=\op{Sym}(V)$ with $\R[x, y]$. Then we have 
\begin{equation}
    \label{eq: R^K for G2}
    R^K = \R[x^2, y^2]. 
\end{equation}
Again, we record some other useful invariant subrings which will arise in future computations: 
\begin{equation}
    \label{eq: invariant subrings for G2}
    R^s = \R\left[ \frac{1}{2}x + \frac{1}{2}y, \left(\frac{3}{2}x-\frac{1}{2}y\right)^2\right], \hspace{2mm} R^t=\R\left[ \frac{3}{2}x+\frac{1}{2}y, \left( - \frac{1}{2}x + \frac{1}{2}y\right)^2\right]. 
\end{equation}

\subsection{The category $\mc{N}_{LV}^0(G_2)$}
\label{sec: The category N_LV for G2}
We start our approach to classifying indecomposable objects in $\mc{N}_{LV}^0(G_2)$ using the same tools as  \S\ref{sec: A type A Lusztig--Vogan category} and \S\ref{sec: A type C Lusztig--Vogan category}. This works at the beginning, but by layer 2, our existing tools are no longer sufficient to complete the computation. We include our initial findings to illustrate at what stage they fail. Then we finish the computation using a computer algorithm. 

\subsubsection{Generators}
\label{sec: generators G2}
The category $\mc{N}_{LV}^0$ has three indecomposable generators:
\[
R, R_s, \text{ and } R_t.
\]

\subsubsection{Layer 1}
\label{sec: layer 1 G2}
Action by $B_s$ and $B_t$ provides six new objects
\[
RB_s, \hspace{2mm} RB_t, \hspace{2mm} R_sB_s, \hspace{2mm} R_s B_t, \hspace{2mm} R_t B_s, \text{ and } R_tB_t.
\]
The isomorphism lemma (Lemma \ref{lem: isomorphism lemma}) implies that two pairs on this list are isomorphic: 
\begin{equation}
    \label{eq: isomorphisms later 1 G2}
    RB_s \simeq R_s B_s \text{ and } R_t B_t \simeq RB_t.
\end{equation}
A computation using the action of $s$ and $t$ on the invariant subrings \eqref{eq: invariant subrings for G2} shows that the splitting lemma (Lemma \ref{lem: splitting lemma}) applies to two modules on the list:
\begin{equation}
    \label{eq: layer 1 splittings G2}
    R_sB_t \simeq R_s^{\oplus v^{-1} + v} \text{ and } R_tB_s \simeq R_t^{\oplus v^{-1} + v}.
\end{equation}

The remaining two objects $RB_s$ and $RB_t$ can be seen to be indecomposable by Lemma \ref{lem: indecomposability lemma}. Note that neither the easy indecomposability criterion (Lemma \ref{lem: indecomposability via generation by homogeneous elt}) nor our monoidal category techniques in \S\ref{sec: layer 1 Sp(1,1)} can be used to establish indecomposability in this example, since $RB_s$ and $RB_t$ are not generated by $1 \otimes 1$ and the splitting lemma does not apply to either $RB_s$ or $RB_t$. This stems from the fact that the reflection subgroup $W_K$ has no simple generators. 

\subsubsection{Layer 2} 
\label{sec: layer 2 G2} 
The only two indecomposable objects obtained from layer 1 are the restricted Bott--Samelson bimodules $RB_s$ and $RB_t$. Hence the only indecomposable objects in $\mc{N}_{LV}^0$ which are not in $\mathrm{Res}^R_{R^K} \SBim$ are the generators $R_s$ and $R_t$.

Unfortunately, this observation doesn't gain us as much traction as one might hope. Here is what we can deduce about the layer 2 restricted Bott--Samelson bimodules $RB_sB_t$ and $RB_tB_s$:
\begin{enumerate}
    \item The bimodule $RB_sB_t$ has a summand of $R_s$, and $RB_tB_s$ has a summand of $R_t$. This can be shown directly by writing down inclusions and retractions, with maps very similar to those in \eqref{eq: B_big ses Sp(1,1)} and \eqref{eq: splitting of R and R_t from RB_tB_s in Sp(1,1) 1}-\eqref{eq: splitting of R and R_t from RB_tB_s in Sp(1,1) 3}.
    \item The resulting summands have the same standard factors, but are not isomorphic. This can be shown using the monoidal structure of $\SBim$, with similar arguments to those in \S\ref{sec: layer 2 Sp(1,1)}.
    \item Denote by $B_1$ the non-standard summand of $RB_sB_t$ and $B_2$ the non-standard summand of $RB_tB_s$. Then $B_1$ does not have $R$ or $R_t$ as a summand, and $B_2$ does not have $R$ or $R_s$ as a summand. This can be shown by assuming such a decomposition exists and deriving a contradiction using the monoidal structure of $\SBim$, similarly to the arguments in \S\ref{sec: layer 2 Sp(1,1)}.  
\end{enumerate}

The facts (1)-(3) above are not enough to conclude that $B_1$ and $B_2$ are indecomposable, but they are all we can obtain using our current set of tools. There are two key differences in this example from the previous ones which make our tools insufficient: 
\begin{enumerate}
\item The size of the categories involved: $\SBim(G_2)$ has 14 indecomposable objects, whereas $\SBim(C_2)$ has 8, and $\SBim(A_2)$ has 6.
\item The smallest restricted Bott--Samelson bimodule containing $B_{big}$: In both  \S\ref{sec: A type A Lusztig--Vogan category} and \S\ref{sec: A type C Lusztig--Vogan category}, the Bott--Samelson $RB_tB_s$ contains $B_{big}$ as a summand, whereas in the current example, the smallest Bott--Samelson bimodule containing $B_{big}$ as a summand is $RB_sB_tB_sB_t$. 
\end{enumerate}
Our tools in \S\ref{sec: four useful lemmas} give us control over the lower layers of the Lusztig--Vogan categories, and the location of $B_{big}$ in $\mathrm{Res}^R_{R^K} \SBim$ allows us to decompose the upper layers. However, when there is too much space in ``the middle'' (i.e. between layer 2 and the layer containing $B_{big}$), our tools fail. 

To finish this computation, we implement an algorithm to compute indecomposable objects. This algorithm is described in the final section of the paper, and the application of the algorithm to split $G_2$ is in \S\ref{sec: Implementation for split G2}.

\section{An algorithm to compute indecomposable objects}
\label{sec: an algorithm to compute indecomposable objects}

In this section, we describe a simple algorithm for finding all indecomposable objects in a finite rank Lusztig--Vogan category. We provide an implementation of this algorithm in Magma. We finish by displaying the $W$-graph of the Lusztig--Vogan categories corresponding to split $G_2$ and $\Sp_4(\R)$ obtained from our Magma implementation of the algorithm.

\subsection{Overview}
Our algorithm follows the recipe for finding indecomposables outlined in \S\ref{sec: a recipe for finding indecomposable objects} and used in \S\ref{sec: A type A Lusztig--Vogan category} and \S\ref{sec: A type C Lusztig--Vogan category}; that is, taking the standard generators, tensoring with generating Soergel bimodules, decomposing, and repeating the process on the indecomposable summands that emerge.  In each of these sections, we found that understanding morphisms in the category was an important tool, but the morphisms were difficult to compute by hand as the module size increased. 

We are able to get around this difficulty in our algorithm by using the crucial fact that the modules in our categories are {\em free as right $R$-modules}. This means that right $R$-module morphisms between them are $R$-matrices, and the graded $(R^K,R)$-bimodule morphism property can be realised as polynomial equations in the entries.
In a similar way, checking that a morphism is an isomorphism and finding direct summands of a module can also be formulated in terms of solutions to polynomial equations. With this perspective, a computer can easily work with morphisms in our categories.

\subsection{The Algorithm}
\label{sec: the algorithm}
The input for the algorithm is a finite Coxeter group $W$ and a reflection subgroup $W_K\subseteq W$. In particular, $W$ does not need to be a Weyl group, so the algorithm also works on Lusztig--Vogan categories which do not correspond to real reductive groups. 

\subsubsection{Encoding objects}
\label{sec: encoding objects}We begin by discussing how to encode an object in our category. The generating objects of $\mathcal{N}^0_{LV}$ and $\SBim$ are graded free of finite rank as right $R$-modules, hence so too is every object in $\mc{N}_{LV}^0$. Given an object $Q$, fix a graded basis $\{b_1, ..., b_m\}$ as a graded right $R$-module. Each element $f \in R^K$ can be seen as a morphism $(f \cdot : g \mapsto f g) \in \Hom_{\textrm{gMod-}R}(Q, Q(\op{deg}(f)))$, and hence as a matrix $A_f$ in $R$ with respect to the chosen basis of $Q$. Explicitly, the $(i, j)$-th entry of $A_f$ is given by the right $R$-coefficient of $b_i$ in $f b_j$. By the Chevalley--Shephard--Todd theorem \cite{Chevalley, ShephardTodd}, $R^K$ is finitely generated\footnote{The finite generation of $R^K$ is the necessary finiteness condition which makes our algorithm run. Hence the algorithm could also be applied to an infinite Coxeter group and finite-index reflection subgroup as long as $R^K$ is finitely generated. Though in such an example, we do not expect $\mc{N}^0_{LV}$ to have finitely many indecomposable objects (up to shift), so we would not expect the algorithm to terminate.}, and it is enough to compute the action matrices for the generators rather than all of $R^K$. Since these left action matrices implicitly encode the basis, we need only remember the graded degree of each basis element. Thus an object $Q$ can be encoded by the following (finite) set of data:
\begin{enumerate}
    \item graded degrees for a graded right $R$-module basis, together with
    \item a graded right $R$-module endomorphism of $Q$ for each generator of $R^K$.
\end{enumerate}
For example, consider the standard $(R^K, R)$-bimodule $R_x$, for $x \in {}^KW$. As a graded right $R$-module, $R_x$ is generated by the degree $0$ element $b_1 = 1$, and by definition we have $f \cdot b_1 = b_1 \cdot x^{-1}(f)$ for all $f \in R^K$. Thus $R_x$ is encoded by the graded degree $\{0\}$ together with the $1\times 1$ left action matrices $A_f = x^{-1}(f)$.

\subsubsection{Carrying the encoding through tensor products}
\label{sec: carrying the encoding through tensor products}
Following the strategy outlined in \S\ref{sec: a recipe for finding indecomposable objects}, we want to carry this encoding through tensoring by $B_s$ for $s \in S$. As a right $R$-module, $B_s \coloneqq R \otimes_{R^s} R(1)$ is free and admits a basis $\{c_{-1} \coloneqq 1 \otimes 1, c_1 \coloneqq \frac{1}{2} (\alpha_s \otimes 1 + 1 \otimes \alpha_s) \}$, with $\deg(c_{\pm 1}) = \pm 1$. Given an object $Q$ with graded basis $\{b_1, ..., b_m\}$, we obtain a graded basis for $Q B_s$ given by $\{b_i \otimes c_{\pm 1} : i \in \{1,...,m\}\}$ ordered lexicographically, and where $b_i \otimes c_{\pm 1}$ has degree $\deg(b_i) \pm 1$. 

For $f \in R$, the left action matrix for $f$ on $B_s$ is given by the graded right $R$-module morphism 
\[ 
\label{action of }
\begin{pmatrix}
    s(f) & 0 \\ \partial_s(f) & f
\end{pmatrix}
: B_s \to B_s(\op{deg}(f))
\]
where $\partial_s$ is the Demazure operator \eqref{eq: demazure operator}. This corresponds to
\[
    f \cdot(c_{-1} \cdot x + c_{+1} \cdot y) = c_{-1} \cdot x s(f) + c_{+1} \cdot (x\partial_s(f) + yf),
\]
for $x, y \in R$. Now, given a collection of left actions $\{A_f : Q \to Q(\op{deg}(f)) \mid f \in R^K\}$ on $Q$, the left action of $f \in R^K$ on the tensor product $Q B_s$ is given by the block matrices
\begin{equation}
\label{eq: QBs action matrices}
    \begin{pmatrix}
        s(A_f) & 0 \\ \partial_s(A_f) & A_f
    \end{pmatrix}
    : QB_s \to QB_s(\op{deg}(f)).
\end{equation}

Thus, given an object $Q \in \mc{N}_{LV}^0$, encoded by: (1) the set of degrees $\{\deg(b_1), \ldots, \deg(b_m)\}$, and (2) the matrices $\{A_f \mid f \in R^K\}$, one can encode the object $QB_s$ by: (1) the set of degrees $\{\deg(b_i) \pm 1 : i \in \{1 \, \ldots, m\}\}$, and (2) the matrices \ref{eq: QBs action matrices}.

\subsubsection{Finding direct summands}
\label{sec: finding direct summands}
We now look to decompose $QB_s$ into its indecomposable summands. Since direct summands correspond exactly to idempotent endomorphisms, this amounts to searching for endomorphisms $M : QB_s \to QB_s$ satisfying $M^2 = M$. An idempotent is primitive precisely when its image is indecomposable. 

We first look to encode arbitrary morphisms in $\mathcal{N}^0_{LV}$. Let $Q$ and $Q'$ be objects in $\mc{N}_{LV}^0$ with ordered graded right $R$-basis $\{b_1,..., b_m\}$ and $\{b_1', ..., b_n'\}$, respectively. 
Let $M \in \Hom_{(R^K,R)\textrm{-gBim}}(Q,Q')$. By forgetting the left $R^K$-action, $M$ can be seen as a matrix with entries in $R$, sending $b_i$ to an element of degree $\deg(b_i)$. In particular, the $(i, j)$-th entry, corresponding to the coefficient of $b_i'$ in $M(b_j)$, is of degree $\deg(b_j) - \deg(b_i')$. Note that each graded piece of $R$ is finitely generated, so each entry is a homogeneous polynomial of the corresponding degree.
Next, for $M : Q \to Q'$ to be a homomorphism of graded $(R^K,R)$-bimodules, it must commute with the left $R^K$-action. Let $Q$ and $Q'$ have left $R^K$-actions given by $\{A_f : f \in R^K\}$ and $\{A_f' : f \in R^K\}$, respectively. This bimodule condition is then $A_f' \circ M = M \circ A_f$ for all $f \in R^K$.

Thus, an idempotent endomorphism $M \in \End_{\mc{N}_{LV}^0}(Q)$ is characterised by the conditions
\[
\begin{cases}
    M^2 - M = 0, \\
    A_f \circ M - M \circ A_f = 0,
\end{cases}
\]
where $f$ runs over generators of $R^K$.
Taken entry-wise, this is a system of polynomial equations whose variables are coefficients of the homogeneous polynomials in $R$ of appropriate degree. 

The solutions to this system are precisely the idempotents in $\End_{\mc{N}_{LV}^0}(Q)$. These solutions are the points of the variety corresponding to the ideal generated by the equations. Thus decomposing $Q$ into a direct sum of indecomposables corresponds to selecting a maximal set of orthogonal primitive idempotents from the variety of idempotents. There are many such sets, but the fact that our category is Krull--Schmidt guarantees that any choice will result in the same list of indecomposables, up to isomorphism. 

We select a set of orthogonal primitive idempotents using the primary decomposition of the ideal generated by the equations. The primary decomposition of the ideal determines the finite set of irreducible components of the variety of idempotents. Idempotents within the same irreducible component must have the same rank\footnote{This is because the rank of an idempotent is given by its trace. Trace is a regular function with values in a finite set, so it must be constant on an irreducible component.}. We start by selecting an irreducible component of minimal non-zero rank, denoted $V$, and choosing an idempotent in this component. This idempotent is guaranteed to be primitive. We add it to a running list. Next, we restrict to the subvariety of $V$ obtained by adding the orthogonality constraint. We perform a primary decomposition on the corresponding ideal and repeat the process. We continue until there are no remaining idempotents in $V$ orthogonal to the list. We then repeat with another irreducible component of minimal non-zero rank, continuing until the sum of the ranks of the orthogonal primitive idempotents on our running list matches the rank of the module. Because our modules are of finite rank, this process eventually terminates with a complete set of orthogonal primitive idempotents.

In order to repeat the tensoring and decomposition process with a new indecomposable appearing as a summand of $Q B_s$, we must translate the data of a primitive idempotent into a graded basis together with left action matrices. This can be done by taking a basis for the image $\overline{Q} \subseteq QB_s$ of the idempotent and setting the left action matrices to be $\{\pi \circ A_f \circ \iota\}$, where $\iota : \overline{Q} \to Q B_s$ and  $\pi : Q B_s \to \overline{Q}$ are an inclusion and a projection with respect to this basis of $\overline{Q}$. Again, it is enough to record only the degrees of the basis elements, as our process builds our choice of basis into the action matrices themselves.

\subsubsection{Enumerating indecomposable objects}
\label{sec: enumerating indecomposable objects}
With the details above, the algorithm to find all indecomposable objects is as follows.
\begin{enumerate}
    \item Start with generators $\{R_x : x \in {}^K W\}$ of $\mathcal{N}_{LV}^0$. Each of these generators is generated by $1 \in R_x$, hence is indecomposable by Lemma \ref{lem: indecomposability via generation by homogeneous elt}.
    \item Write $\{Q_i\}_{i \in I}$ for the running list of indecomposables. If this list does not exhaust all indecomposables, then another must exist as a summand of $Q_iB_s$ for some $i \in I$ and $s \in S$. Thus, decompose every object of the form $Q_i B_s$ by primitive idempotents and encode the resulting summands as in \S\ref{sec: finding direct summands}.
    \item Compare the new indecomposables with the running list to see if any new ones have appeared. Two objects can be compared by looking for invertible graded bimodule morphisms between them. Using the above notation, an invertible graded bimodule morphism $M: Q \to Q'$ is given by solving the system of polynomial equations 
    \[
    \begin{cases}
        A_f' \circ M - M \circ A_f = 0, \\
        \det(M)t - 1 = 0,
    \end{cases}
    \]
    for each generator $f \in R^K$. The last equation encodes invertibility of $M$ by an auxiliary variable $t \in \mathbb{Q}$.
    \item If a new indecomposable appears in (3), return to step (2). Repeat until (3) returns no new indecomposables.
\end{enumerate}

Because we are assuming $W$ to be finite, $\SBim(W)$ has finitely many indecomposable objects (up to shift). As Lusztig--Vogan categories have finitely many generators and are Krull--Schmidt, they too will have finitely many indecomposable objects (up to shift) in this setting. Hence, the algorithm above terminates with a complete list of indecomposable objects in $\mc{N}_{LV}^0$. 

Note that our description of the algorithm above is formulated for clarity rather than efficiency.

\subsection{Magma Implementation}
\label{sec: magma implementation}
We provide an implementation of this algorithm in the computer algebra system Magma \cite{Magma}. The source code can be found at \cite{magma-code}. Our program relies on solving multivariate polynomial systems and finding a complete set of orthogonal primitive idempotents within the solution sets. We accomplish this in Magma by determining primary decompositions of the relevant ideals. 
The only other functionality we benefit from in Magma is its implementations of Coxeter groups and finitely-presented groups, which we use for automating the construction of the geometric realisation and ${}^KW$. Specifically, we use the implementation of Coxeter groups as matrix groups for the former and the implementation of Coxeter groups as abstract finitely-presented groups for the latter. 

Our algorithm has been tested on all Lusztig--Vogan categories of ranks $1$, $2$ and $3$, as well as those corresponding to the rank $4$ form $\SU(4, 1)$ and the rank $5$ form $\SU(5, 1)$. In examples where $(W_K,W)$ corresponds to a real reductive group, the algorithm returns the correct $W$-graph of the corresponding Lusztig--Vogan module. 

\subsection{Limitations}
\label{sec: limitations}

The core of the Magma implementation is the computation of Gr\"obner bases during primary decomposition. In general, the computation of Gr\"obner bases is known to be EXPSPACE-complete \cite{grobner-expspace-hard}, hence, for examples with a large number of generators and dimension, we expect a significant slowdown and large memory usage during computation. 

\subsection{Implementation for split $G_2$}
\label{sec: Implementation for split G2}
We are able to complete our search for indecomposables in the Lusztig--Vogan category corresponding to split $G_2$ with the Magma implementation in \S\ref{sec: magma implementation}. Here, as in \S\ref{sec: A type G Lusztig--Vogan category}, we have $W = D_6$ and $W_K = \langle sts,tst \rangle$. The resulting $W$-graph is displayed in Figure \ref{fig: W-graph of G_2}.

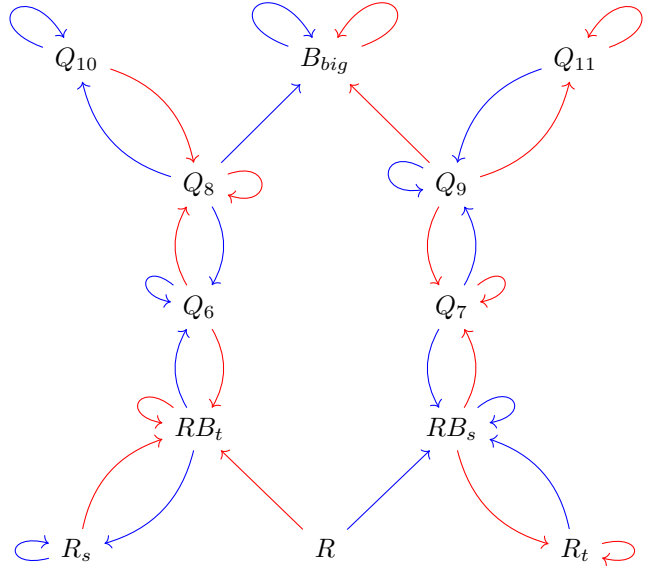
\begin{figure}[H]
\[
\begin{tikzcd}[row sep=3em, column sep=2em]
Q_{10} \arrow[loop left, blue, out=160, in=120, looseness=8] \arrow[dr, bend left, red] && B_{big} \arrow[loop left, blue, out=160, in=120, looseness=8] \arrow[loop right, red, out=20, in=60, looseness=8] && Q_{11} \arrow[loop right, red, out=20, in=60, looseness=8] \arrow[dl, blue, bend right] \\
& Q_8  \arrow[ul, bend left, blue]  \arrow[loop left, red, out=20, in=-20, looseness=6] \arrow[d, bend left, blue] \arrow[ur, blue] & & Q_9 \arrow[loop left, blue, out=150, in=190, looseness=6] \arrow[d, bend right, red] \arrow[ur, bend right, red] \arrow[ul, red]& \\
& Q_6 \arrow[u, bend left, red]  \arrow[loop left, blue, out=140, in=170, looseness=6] \arrow[d, bend left, red] & & Q_7 \arrow[loop right, red, out=40, in=10, looseness=6] \arrow[d, bend right, blue] \arrow[u, bend right, blue]  & \\
& RB_t \arrow[u, bend left, blue]  \arrow[loop left, red, out=140, in=170, looseness=6] \arrow[dl, bend left, blue]  & & RB_s \arrow[loop right, blue, out=40, in=10, looseness=6] \arrow[u, bend right, red] \arrow[dr, bend right, red] &\\
R_s \arrow[loop left, looseness=8, blue] \arrow[ur, bend left, red]& & R \arrow[ul, red] \arrow[ur, blue] & &  R_t \arrow[loop right, looseness=8, red] \arrow[ul, bend right, blue] 
\end{tikzcd}
\]
\caption{A picture of $\mc{N}^0_{LV}(G_2)$. Nodes are indecomposable objects. A red (resp. blue) arrow from object $Q$ to object $Q'$ indicates that $Q'$ is a direct summand of $QB_s$ (resp. $QB_t$). }
\label{fig: W-graph of G_2}
\end{figure}

\subsection{Implementation for $\Sp_4(\R)$}
\label{sec: Implementation for Sp4(R)}
The remaining equal rank 2 real group is the group $\Sp_4(\R)$. For brevity, we will not describe the Lie theoretic structure of this group. (It is similar to the group $\Sp(1,1)$ in \S\ref{sec: A type C Lusztig--Vogan category}.) In this example, $W = D_4$ and $W_K = \langle s \rangle$. An application of the Magma implementation in \S\ref{sec: magma implementation} returns the $W$-graph in Figure \ref{fig: W-graph of Sp4(R)}.

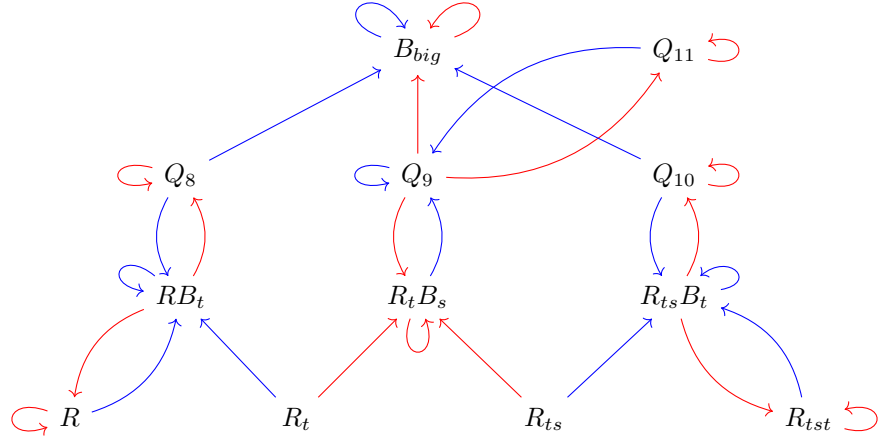
\begin{figure}[H]
\[
\begin{tikzcd}[row sep=3em, column sep=2em]
& & & B_{big}\arrow[loop left, blue, out=160, in=120, looseness=6] \arrow[loop right, red, out=20, in=60, looseness=6] & & Q_{11}\arrow[loop right, red, out=345, in=375, looseness=6]\arrow[dll, bend right, blue]\\
& Q_8\arrow[loop left, red, out=165, in=195, looseness=8]\arrow[d, bend right, blue]\arrow[urr, blue] & & Q_9\arrow[loop left, blue, out=165, in=195, looseness=8]\arrow[d, bend right, red]\arrow[u, red]\arrow[urr, bend right, red] & & Q_{10}\arrow[loop right, red, out=345, in=375, looseness=6]\arrow[d, bend right, blue]\arrow[ull, blue]\\
& RB_t\arrow[loop left, blue, out=140, in=170, looseness=6]\arrow[dl, bend right, red]\arrow[u, bend right, red] & & R_tB_s\arrow[loop left, red, out=250, in=290, looseness=8]\arrow[u, bend right, blue] & & R_{ts}B_t \arrow[loop right, blue, out=10, in=40, looseness=4]\arrow[dr, bend right, red]\arrow[u, bend right, red]\\
R \arrow[loop left, red, out=160, in=200, looseness=8]\arrow[ur, bend right, blue] & & R_t\arrow[ul, blue]\arrow[ur, red] & &  R_{ts}\arrow[ul, red]\arrow[ur, blue] & &  R_{tst} \arrow[loop right, red, out=345, in=375, looseness=6]\arrow[ul, bend right, blue]
\end{tikzcd}
\]
\caption{A picture of $\mc{N}^0_{LV}(\Sp_4(\R))$. Nodes are indecomposable objects. A red (resp. blue) arrow from object $Q$ to object $Q'$ indicates that $Q'$ is a direct summand of $QB_s$ (resp. $QB_t$). }
\label{fig: W-graph of Sp4(R)}
\end{figure}

\bibliographystyle{alpha}
\bibliography{SBim}

\end{document}